\documentclass[12pt]{amsart}
\usepackage{amssymb, amscd}

\begin{document}

\title[\null]
{Extension of twisted Hodge metrics \\
for K\"ahler morphisms}

\author[\null]{Christophe Mourougane and Shigeharu TAKAYAMA}

\maketitle
\baselineskip=18pt


\theoremstyle{plain}
  \newtheorem{thm}{Theorem}[section]
  \newtheorem{main}[thm]{Main Theorem}
  \newtheorem{defthm}[thm]{Definition-Theorem}
  \newtheorem{prop}[thm]{Proposition}
  \newtheorem{lem}[thm]{Lemma}
  \newtheorem{cor}[thm]{Corollary}
  \newtheorem{conj}[thm]{Conjecture}
  \newtheorem{sublem}[thm]{Sublemma}
  \newtheorem{mainlem}[thm]{Main Lemma}
  \newtheorem{variant}[thm]{Variant}
\theoremstyle{definition}
  \newtheorem{dfn}[thm]{Definition}
  \newtheorem{exmp}[thm]{Example}
  \newtheorem{co-exmp}[thm]{Counter-Example}
  \newtheorem{prob}[thm]{Problem}
  \newtheorem{notation}[thm]{Notation}
  \newtheorem{quest}[thm]{Question}
  \newtheorem{setup}[thm]{Set up}
  \newtheorem{assum}[thm]{Assumption}
\theoremstyle{remark}
  \newtheorem{rem}[thm]{Remark}
  \newtheorem{com}[thm]{Comment}
\renewcommand{\theequation}{\thesection.\arabic{equation}}
\setcounter{equation}{0}

\newcommand{\BC}{{\mathbb{C}}}
\newcommand{\BN}{{\mathbb{N}}}
\newcommand{\BP}{{\mathbb{P}}}
\newcommand{\BQ}{{\mathbb{Q}}}
\newcommand{\BR}{{\mathbb{R}}}

\newcommand{\CalC}{{\mathcal{C}}}
\newcommand{\CalD}{{\mathcal{D}}}
\newcommand{\CE}{{\mathcal{E}}}
\newcommand{\CF}{{\mathcal{F}}}
\newcommand{\CH}{{\mathcal{H}}}
\newcommand{\CI}{{\mathcal{I}}}
\newcommand{\CJ}{{\mathcal{J}}}
\newcommand{\CK}{{\mathcal{K}}}
\newcommand{\CL}{{\mathcal{L}}}
\newcommand{\CM}{{\mathcal{M}}}
\newcommand{\CO}{{\mathcal{O}}}
\newcommand{\CS}{{\mathcal{S}}}
\newcommand{\CU}{{\mathcal{U}}}

\newcommand{\fm}{{\mathfrak{m}}}

\newcommand\ga{\alpha}
\newcommand\gb{\beta}
\def\th{\theta}
\newcommand\vth{\vartheta}
\newcommand\Th{\Theta}
\newcommand\ep{\varepsilon}
\newcommand\Ga{\Gamma}
\newcommand\w{\omega}
\newcommand\Om{\Omega}
\newcommand\Dl{\Delta}
\newcommand\del{\delta}
\newcommand\sg{\sigma}
\newcommand\vph{\varphi}
\newcommand\lam{\lambda}
\newcommand\Lam{\Lambda}

\newcommand\lra{\longrightarrow}
\newcommand\ra{\rightarrow}
\newcommand\ot{\otimes}
\newcommand\wed{\wedge}
\newcommand\ol{\overline}
\newcommand\isom{\, \wtil{\lra}\, }
\newcommand\sm{\setminus}

\newcommand\ai{\sqrt{-1}}
\newcommand\rd{{\partial}}
\newcommand\rdb{{\overline{\partial}}}
\newcommand\levi{\ai\rd\rdb}

\newcommand\wtil{\widetilde}
\newcommand\tf{\widetilde f}
\newcommand\tX{\widetilde X}
\newcommand\tx{\widetilde x}
\newcommand\tw{\widetilde \w}

\newcommand\what{\widehat}
\newcommand\hsg{\widehat \sg}

\newcommand\codim{\mbox{{\rm codim}}}
\newcommand\End{\mbox{{\rm End}}\, }
\newcommand\Exc{\mbox{{\rm Exc}}\, }
\newcommand\im{\mbox{{\rm im}}\, }
\newcommand\Ima{\mbox{{\rm Im}}\, }
\newcommand\Ker{\mbox{{\rm Ker}}\, }
\newcommand\Reg{\mbox{{\rm Reg}}\, }
\newcommand\Sing{\mbox{{\rm Sing}}\, }
\newcommand\Supp{\mbox{{\rm Supp}}\, }

\newcommand\Rq{R^qf_*(K_{X/Y} \ot E)}
\newcommand\F{R^qf_*(K_{X/Y} \ot E)}
\newcommand\FF{R^qf''_*(K_{X''/Y'} \ot E'')}

\newcommand\Xo{{X^\circ}}
\newcommand\Uo{{U^\circ}}
\newcommand\fo{{f^\circ}}
\newcommand\tauo{{\tau^\circ}}
\newcommand\go{{g_{\CO(1)}^\circ}}


\section{Introduction}

The subject in this paper is the positivity of direct image sheaves 
of adjoint bundles $\Rq$, for a K\"ahler morphism $f : X \lra Y$ endowed 
with a Nakano semi-positive holomorphic vector bundle $(E, h)$ on $X$.
In our previous paper \cite{MT2}, 
generalizing a result \cite{B} in case $q = 0$,
we obtained the Nakano semi-positivity of $\Rq$ with respect to 
a canonically attached metric, the so-called Hodge metric,
under the assumption that $f : X \lra Y$ is smooth.
However the smoothness assumption on $f$ is rather restrictive,
and it is desirable to remove it.
This is the aim of this paper.

To state our result precisely, let us fix notations and recall basic facts.
Let $f : X \lra Y$ be a holomorphic map of complex manifolds.
A real $d$-closed $(1,1)$-form $\w$ on $X$ is said to be 
{\it a relative K\"ahler form} for $f$, if for every point $y \in Y$, 
there exists an open neighbourhood $W$ of $y$ and 
a smooth plurisubharmonic function $\psi$ on $W$ such that 
$\w + f^*(\ai\rd\rdb \psi)$ is a K\"ahler form on $f^{-1}(W)$.
A morphism $f$ is said to be {\it K\"ahler}, if there exists 
a relative K\"ahler form for $f$ (\cite[6.1]{Tk}),
and $f : X \lra Y$ is said to be a {\it K\"ahler fiber space},
if $f$ is proper, K\"ahler, and surjective with connected fibers.

\begin{setup} \label{basic}
(General global setting.) \ 
(1)
Let $X$ and $Y$ be complex manifolds of $\dim X = n + m$ and 
$\dim Y = m$, and let $f : X \lra Y$ be a K\"ahler fiber space.
We do not fix a relative K\"ahler form for $f$, unless otherwise stated.
The {\it discriminant locus} of $f : X \lra Y$ is the minimum closed 
analytic subset $\Dl \subset Y$ such that $f$ is smooth over $Y \sm \Dl$.

(2)
Let $(E, h)$ be a Nakano semi-positive holomorphic vector bundle on $X$.
Let $q$ be an integer with $0 \le q \le n$.
By Koll\'ar \cite{Ko1} and Takegoshi \cite{Tk}, 
$\F$ is torsion free on $Y$,
and moreover it is locally free on $Y \sm \Dl$ (\cite[4.9]{MT2}).
In particular we can let $S_q \subset \Dl$ be 
the minimum closed analytic subset of $\codim_Y S_q \ge 2$ such that 
$\F$ is locally free on $Y \sm S_q$.
Let $\pi : \BP(\F|_{Y \sm S_q}) \lra Y \sm S_q$
be the projective space bundle, and let $\pi^*(\F|_{Y \sm S_q}) \lra \CO(1)$
be the universal quotient line bundle.

(3)
Let $\w_f$ be a relative K\"ahler form for $f$.
Then we have the Hodge metric $g$ on the vector bundle 
$\F|_{Y \sm \Dl}$ with respect to $\w_f$ and $h$ (\cite[\S 5.1]{MT2}).
By the quotient 
$\pi^*(\Rq|_{Y \sm \Dl}) \lra \CO(1)|_{\pi^{-1}(Y \sm \Dl)}$,
the metric $\pi^*g$ gives the quotient metric $\go$ on 
$\CO(1)|_{\pi^{-1}(Y \sm \Dl)}$.
The Nakano, even weaker Griffiths, semi-positivity of $g$
(by \cite[1.2]{B} for $q=0$, and by \cite[1.1]{MT2} for $q$ general) 
implies that $\go$ has a semi-positive curvature.
\qed
\end{setup}

In these notations, our main result is as follows
(see also \S \ref{Variants} for some variants).

\begin{thm} \label{MT}
Let $f : X \lra Y$, $(E, h)$ and $0 \le q \le n$ be as in Set up \ref{basic}.

(1) Unpolarized case.
Then, for every relatively compact open subset $Y_0 \subset Y$,
the line bundle $\CO(1)|_{\pi^{-1}(Y_0 \sm S_q)}$ 
on $\BP(\Rq|_{Y_0 \sm S_q})$ has a singular Hermitian metric 
with semi-positive curvature, and which is smooth on $\pi^{-1}(Y_0 \sm \Dl)$.

(2) Polarized case.
Let $\w_f$ be a relative K\"ahler form for $f$.
Assume that there exists a closed analytic set 
$Z \subset \Dl$ of $\codim_Y Z \ge 2$ such that 
$f^{-1}(\Dl)|_{X \sm f^{-1}(Z)}$ is a divisor
and has a simple normal crossing support (or empty).
Then the Hermitian metric $\go$ on $\CO(1)|_{\pi^{-1}(Y \sm \Dl)}$
can be extended as a singular 
Hermitian metric $g_{\CO(1)}$ with semi-positive curvature
of $\CO(1)$ on $\BP(\Rq|_{Y \sm S_q})$. 
\end{thm}

If in particular in Theorem \ref{MT}, 
$\Rq$ is locally free and $Y$ is a smooth
projective variety, then the vector bundle $\Rq$ is pseudo-effective
in the sense of \cite[\S 6]{DPS}.
%
The above curvature property of $\CO(1)$ leads to the following
algebraic positivity of $\Rq$.

\begin{thm} \label{algebraic}
Let $f : X \lra Y$ be a surjective morphism with connected fibers
between smooth projective varieties,
and let $(E, h)$ be a Nakano semi-positive holomorphic vector bundle 
on $X$.
Then the torsion free sheaf $\Rq$ is weakly positive over $Y \sm \Dl$
(the smooth locus of $f$), in the sense of Viehweg \cite[2.13]{Vi2}.
\end{thm}

Here is a brief history of the semi-positivity of direct image
sheaves, especially in case the map $f : X \lra Y$ is not smooth.
The origin is due to Fujita \cite{Ft} for $f_*K_{X/Y}$ over a curve,
in which he analyzed the singularities of the Hodge metric.
After \cite{Ft}, there are a lot of works mostly in algebraic geometry
to try to generalize \cite{Ft}, for example
by Kawamata \cite{Ka1}\,\cite{Ka2}\,\cite{Ka3}, Viehweg \cite{Vi1},
Zucker \cite{Z}, Nakayama \cite{N1}, Moriwaki \cite{Mw}, 
Fujino \cite{Fn}, Campana \cite{C}.
Their methods heavily depend on the theory of a variation 
of Hodge structures.
While Koll\'ar \cite{Ko1} and Ohsawa \cite[\S 3]{Oh} reduce 
the semi-positivity to their vanishing theorems.
We refer to \cite{EV}\,\cite[V.\S 3]{N2}\,\cite{Vi2} for further
related works.
There are more recent related works from the Bergman kernel point of view,
by Berndtsson-P\u{a}un \cite{BP1}\,\cite{BP2} and Tsuji \cite{Ts}.
Their interests are the positivity of a relative canonical bundle
twisted with a line bundle with a singular Hermitian metric of
semi-positive curvature, or its zero-th direct image,
which are slightly different from ours in this paper.

The position of this paper is rather close to 
the original work of Fujita.
We work in the category of K\"ahlerian geometry.
We will prove that a Hodge metric defined over $Y \sm \Dl$
can be extended across the discriminant locus $\Dl$, 
which is a local question on the base.
Because of the twist with a Nakano semi-positive vector bundle $E$
which may not be semi-ample, 
one can not take (nor reduce a study to) 
the variation of Hodge structures approach.
The algebraic approach quoted above only concludes that
the direct image sheaves have algebraic semi-positivities, such as
nefness, or weak positivity.
It is like semi-positivity of integration of the curvature along
subvarieties.
These algebraic semi-positivities already requires a global property 
on the base, for example (quasi-)projectivity.
In the algebraic approach, to obtain a stronger result,
they sometimes pose a normal crossing condition of the discriminant locus
$\Dl \subset Y$ of the map, and/or a unipotency of local monodromies.
We are free from these conditions, but we must admit that our method 
does not tell local freeness nor nefness of direct images sheaves.
We really deal with Hodge metrics, and 
we do not use the theory of a variation of Hodge structures, 
nor global geometry on the base, in contrast to the algebraic approach.

In connection with a moduli or a deformation theory,
a direct image sheaf on a parameter space defines a canonically
attached sheaf quite often, and then the curvature of the Hodge metric
describes the geometry of the parameter space.
Then, especially as a consequence of our previous paper \cite{MT2},
the Nakano semi-positivity of the curvature on which the family is smooth,
is quite useful in practice.
If there exists a reasonable compactification of the parameter space,
our results in this paper can be applied to obtain boundary properties.
There might be further applications in this direction, we hope.
While the algebraic semi-positivity is more or less Griffiths 
semi-positivity, which has nice functorial properties
but is not strong enough especially in geometry.

Our method of proof is to try to generalize the one in \cite{Ft}.
The main issue is to obtain a positive lower bound of
the singularities of a Hodge metric $g$.
It is like a uniform upper estimate for a family of plurisubharmonic
functions $-\log g(u,u)$ around $\Dl \subset Y$,
where $u$ is any nowhere vanishing local section of $\Rq$.
In case $\dim Y = 1$ and arbitrary $q \ge 0$, 
we can obtain rather easily the results we have stated, 
by combining \cite{Ft} and our previous work \cite{MT2}.
In case $\dim Y \ge 1$, a major difficulty arises.
If the fibers of $f$ are reduced, it is not difficult to apply again
the method we took in case $\dim Y = 1$.
However in general, a singular fiber is not a divisor anymore, 
and in addition it can be non-reduced.
To avoid such an analytically uncomfortable situation,
we employ a standard technique in algebraic geometry;
a semi-stable reduction and an analysis of singularities which
naturally appear in the semi-stable reduction process (\S \ref{3}).
A Hodge metric after a semi-stable reduction would be better
and would be handled by known techniques, because fibers become reduced.
Then the crucial point in the metric analysis is 
a comparison of the original Hodge metric
and a Hodge metric after a semi-stable reduction.
As a result of taking a ramified cover and a resolution of singularities
in a semi-stable reduction,
we naturally need to deal with a degenerate K\"ahler form,
and then we are forced to develop a theory of relative harmonic forms
(as in \cite{Tk}) with respect to the degenerate K\"ahler form (\S \ref{4}).
After a series of these observations, we bound singularities of 
the Hodge metric, and obtain a uniform estimate to extend the Hodge metric
(\S \ref{5}).
The proof is not so simple to mention more details here,
because we need to consider a uniform estimate, when
a rank one quotient $\pi_L : \Rq \lra L$ moves and a section of
the kernel of $\pi_L$ moves.
There is a technical introduction \cite{MT3}, where we explain
the case $\dim Y = 1$, or the case where the map $f$ has reduced fibers.

{\it Acknowledgement}.
The second named author would like to express his thanks
for Professor Masanori Ishida for answering questions on
toric geometry.


\section{Hodge Metric} \label{2}

\subsection{Definition of Hodge metric}

Let us start by recalling basic definitions and facts.
Let $f : X \lra Y$ be a K\"ahler fiber space as in Set up \ref{basic}.
For a point $y \in Y \sm \Dl$, we denote by $X_y = f^{-1}(y),
\w_y = \w|_{X_y}, E_y = E|_{X_y}, h_y = h|_{X_y}$,
and for an open subset $W \subset Y$, we denote by $X_W = f^{-1}(W)$.
We set $\Omega_{X/Y}^p = \bigwedge^p (\Om_X^1/(\Ima f^*\Om_Y^1))$
rather formally for the natural map $f^*\Om_Y^1 \lra \Om_X^1$, 
because we will only deal with
$\Omega_{X/Y}^p$ where $f$ is smooth.
For an open subset $U \subset X$ where $f$ is smooth,
and for a differentiable form $\sg \in A^{p,0}(U, E)$, 
we say $\sg$ is {\it relatively holomorphic}
and write $[\sg] \in H^0(U, \Omega_{X/Y}^p \ot E)$,
if for every $x \in U$, 
there exists an coordinate neighbourhood $W$ of $f(x) \in Y$
with a nowhere vanishing $\theta \in H^0(W, K_Y)$ such that 
$\sg \wed f^*\theta \in H^0(U \cap X_W, \Omega_X^{p+m} \ot E)$
(\cite[\S 3.1]{MT2}).

We remind the readers of the following basic facts,
which we will use repeatedly.
See \cite[6.9]{Tk} for more general case when $Y$ may be singular, 
\cite[4.9]{MT2} for (3), and also \cite{Ko1}.

\begin{lem} \label{torsion free}
Let $f : X \lra Y$ and $(E, h)$ be as in Set up \ref{basic}.
Let $q$ be a non-negative integer.
Then 
(1) $\Rq$ is torsion free,  
(2) Grauert-Riemenschneider vanishing:\ 
$\Rq = 0$ for $q > n$, and
(3) $\Rq$ is locally free on $Y \sm \Dl$.
\end{lem}

Using Grauert-Riemenschneider vanishing, a Leray spectral sequence
argument shows that $\Rq$ does not depend on 
smooth bimeromorphic models of $X$.
Choices of a smooth bimeromorphic model of $X$ and of 
a relative K\"ahler form for the new model give rise to
a Hermitian metric on the vector bundle $\Rq|_{Y \sm \Dl}$ 
as follows.

\begin{dfn}  \label{hodge}
(Hodge metric \cite[\S 5.1]{MT2}.) \
In Set up \ref{basic}, assume that $f$ is smooth, 
$Y$ is Stein with $K_Y \cong \CO_Y$  (with a nowhere vanishing 
$\theta_Y \in H^0(Y, K_Y)$), and $X$ is K\"ahler.
A choice of a K\"ahler form $\w$ on $X$ gives an injection
$S_\w := S_f^q : \Rq \lra f_*(\Om_{X/Y}^{n-q} \ot E)$.
Then for every pair of vectors $u_y, v_y \in \Rq_y$, we define
$$
	g(u_y, v_y) 
	= \int_{X_y} (c_{n-q} /q!) \
		\w_y^q \wed S_\w(u_y) \wed h_y\ol{S_\w(v_y)}.
$$
Here $c_p = \ai^{p^2}$ for every integer $p \ge 0$.
Since $f$ is smooth, these pointwise inner products define 
a smooth Hermitian metric $g$ on $\Rq$, 
which we call the {\it Hodge metric} with respect to $\w$ and $h$. 
\qed
\end{dfn}

Details for the construction of the map $S_\w$ will be provided
in Step 2 in the proof of Proposition \ref{Takegoshi}.
In Definition \ref{hodge}, another choice of a K\"ahler form 
$\w'$ on $X$ gives another metric $g'$ on $\Rq$. 
However in case $\w$ and $\w'$ relate with $\w|_{X_y} = \w'|_{X_y}$ 
for any $y \in Y$, these metrics coincide $g = g'$ (\cite[5.2]{MT2}).
Thus a Hodge metric is defined for a polarized smooth K\"ahler 
fiber space in Set up \ref{basic}.
In case when $q = 0$, the Hodge metric does not depend on 
a relative K\"ahler form.
In fact, it is given by
$$
 g(u_y, v_y) 
 = \int_{X_y} c_n u_y \wed h_y\ol{v_y}
$$
for $u_y, v_y \in H^0(X_y, K_{X_y} \ot E_y)$.


\subsection{Localization}

We consider the following local setting, around a codimension 1
general point of $\Dl \subset Y$ 
(possibly after a modification of $X$).

\begin{setup} \label{local}
(Generic local, relative normal crossing setting.) \
Let $f : X \lra Y$, $(E, h)$ and $0 \le q \le n$ be as in Set up \ref{basic}.
Let us assume further the following:\

(1)
The base $Y$ is (biholomorphic to) 
a unit polydisc in $\BC^m$ with coordinates $t = (t_1, \ldots, t_m)$.
Let $K_Y \cong \CO_Y$ be a trivialization by a nowhere vanishing
section $dt = dt_1 \wed \ldots \wed dt_m \in H^0(Y, K_Y)$.

(1.i)
$f$ is flat, and the discriminant locus $\Dl \subset Y$ is 
$\Dl = \{t_m = 0\}$ (or $\Dl = \emptyset$),

(1.ii)
the effective divisor $f^*\Dl$ has a simple normal crossing support,
and

(1.iii)
the morphism $\Supp f^*\Dl \lra \Dl$ is relative
normal crossing (see below).

(2)
$\Rq \cong \CO_Y^{\oplus r}$, i.e.,
globally free and trivialized of rank $r$.

(3)
$X$ admits a K\"ahler form $\w$.
Let $g$ be the Hodge metric on $\Rq|_{Y \sm \Dl}$ with respect to 
$\w$ and $h$.

We may replace $Y$ by slightly smaller polydiscs, or 
may assume everything is defined over a slightly larger polydisc. 
\qed
\end{setup}

In the above, $\Supp f^*\Dl \lra \Dl$ is relative
normal crossing means that, around every $x \in X$, 
there exists a local coordinate $(U; z = (z_1, \ldots, z_{n+m}))$ 
such that $f|_U$ is given by $t_1 = z_{n+1}, \ldots, t_{m-1} = z_{n+m-1}, 
t_m = z_{n+m}^{b_{n+m}} \prod_{j=1}^{n} z_j^{b_j}$ with 
non-negative integers $b_j$ and $b_{n+m}$.

Then the following version of Theorem \ref{MT}\,(2)
is our main technical statement.

\begin{thm} \label{bdd}
Let $f : (X, \w) \lra Y \subset \BC^m$, $(E, h)$ and $0 \le q \le n$ be as in 
Set up \ref{local}.
The pull-back metric $\pi^*g$ of the Hodge metric $g$ on 
$\Rq|_{Y \sm \Dl}$ with respect to $\w$ and $h$ gives 
the quotient metric $\go$ on $\CO(1)|_{\pi^{-1}(Y \sm \Dl)}$.
The smooth Hermitian metric $\go$ 
extends as a singular Hermitian metric $g_{\CO(1)}$ on $\CO(1)$
with semi-positive curvature.
\end{thm}

We shall see our main result:\ Theorem \ref{MT} by 
taking Theorem \ref{bdd} for granted, in the rest of this section.
For a general K\"ahler fiber space $f : X \lra Y$,
we can reduce the study of a Hodge metric to 
the study which is local on $Y$ as in Set up \ref{local},
possibly after taking blowing-ups of $X$.

\begin{lem} \label{decomp}
Let $f : X \lra Y$, $(E, h)$ and $0 \le q \le n$ be as in 
Set up \ref{basic}.
Let $Y_0 \subset Y$ be a relatively compact open subset. 
Let $Z_0 \subset \Dl$ be a closed analytic subset of 
$\codim_Y Z_0 \ge 2$ such that $\Dl \sm Z_0$ is a smooth divisor
(or empty). 
Possibly after restricting everything on a relatively compact 
open neighbourhood over $Y_0$, let $\mu : X' \lra X$ be a birational map
from a complex manifold $X'$, which is obtained by a finite number 
of blowing-ups along non-singular centers, and which is biholomorphic 
over $X \sm f^{-1}(\Dl)$, such that
$f^*(\Dl \sm Z_0)$ is a divisor with simple normal crossing support
on $X \sm f^{-1}(Z_0)$.
Let $\w_{f'}$ be a relative K\"ahler form 
for $f' := f \circ \mu$ over $Y_0$.
Then 

(1)
there exist 
(i) 
a closed analytic subset $Z \subset \Dl$ of $\codim_Y Z \ge 2$,
(ii) 
an open covering $\{W_i\}_i$ of $Y_0 \sm Z$, and
(iii)
a K\"ahler form $\w_i$ on $X'_{W_i} = {f'}^{-1}(W_i)$ for every $i$,
such that 
(a) 
for every $i$, $W_i$ is biholomorphic to the unit polydisc, and 
the induced $f'_i : (X'_{W_i}, \w_i) \lra W_i \subset \BC^m$, 
$(\mu^*E, \mu^*h)|_{X'_{W_i}}$ and $0 \le q \le n$ satisfy 
all the conditions in Set up \ref{local}, and that
(b)
$\w_i|_{X'_y} = \w_{f'}|_{X'_y} $ for every $i$ and $y \in W_i$.
Moreover one can take $\{W_i\}_i$ so that, the same is true, 
even if one replaces all $W_i$ by slightly smaller concentric polydiscs.

(2)
Via the isomorphism $\Rq \cong R^qf'_*(K_{X'/Y} \ot \mu^*E)$,
the Hodge metric on $R^qf'_*(K_{X'/Y} \ot \mu^*E)|_{Y_0 \sm \Dl}$
with respect to $\w_{f'}$ and $\mu^*h$ induces a smooth Hermitian metric 
$g$ with Nakano semi-positive curvature on $\Rq|_{Y_0 \sm \Dl}$. 
\end{lem}

\begin{proof}
In general, a composition $f \circ \mu$ of $f$ and a blow-up
$\mu : X' \lra X$ along a closed complex submanifold of $X$,
is only locally K\"ahler (\cite[6.2.i-ii]{Tk}).
(We do not know if $f \circ \mu$ is K\"ahler.
This is the point, why we need to mention 
``on every relatively compact open subset $Y_0 \subset Y$'' 
in Theorem \ref{MT}\,(1).)
Hence our modification $f' : X' \lra Y$ is locally K\"ahler,
and we can take a relative K\"ahler form $\w_{f'}$ for $f'$ over $Y_0$.
As we explained before, we have
$R^q(f \circ \mu)_*(K_{X'/Y} \ot \mu^*E) = \Rq$ by Lemma \ref{torsion free}. 
%

To see (1), we note Lemma \ref{torsion free} that 
$R^qf'_*(K_{X'/Y} \ot \mu^*E)$ is locally free in codimension 1 on $Y$.
We then take $Z \supset Z_0$ to be the union of all subvarieties 
along which one of (1) -- (2) in Set up \ref{local} fails for $f'$.
Others are almost clear (by construction).
\end{proof}

The following is a more precise statement of Theorem \ref{MT}\,(1).

\begin{prop} \label{MT'}
Let $f : X \lra Y$, $(E, h)$ and $0 \le q \le n$ be as in 
Set up \ref{basic}.
Let $Y_0 \subset Y$ be a relatively compact open subset. 
After taking a modification $\mu : X' \lra X$ 
(on a neighbourhood of $X_0 = f^{-1}(Y_0)$)
and a relative K\"ahler
form $\w_{f'}$ for $f' = f \circ \mu$ over $Y_0$ as in Lemma \ref{decomp},
the Hermitian metric $g$ on $\Rq|_{Y_0 \sm \Dl}$ in Lemma \ref{decomp}\,(2) 
induces the quotient metric $g^\circ_{\CO(1)|_{\pi^{-1}(Y_0 \sm \Dl)}}$ on 
$\CO(1)|_{\pi^{-1}(Y_0 \sm \Dl)}$ with semi-positive curvature.
Then the smooth Hermitian metric 
$g^\circ_{\CO(1)|_{\pi^{-1}(Y_0 \sm \Dl)}}$ 
extends as a singular Hermitian metric 
$g_{\CO(1)|_{\pi^{-1}(Y_0 \sm S_q)}}$ on 
$\CO(1)|_{\pi^{-1}(Y_0 \sm S_q)}$ with semi-positive curvature.
\end{prop}

\begin{proof}[Proof of Theorem \ref{MT}]
(1)
It is enough to show Proposition \ref{MT'}.
We use the notations in Lemma \ref{decomp}.
We apply Theorem \ref{bdd} on each $W_i \subset Y_0 \sm Z$.
Then  we see, at this point, the smooth Hermitian metric $\go$ on 
$\CO(1)|_{\pi^{-1}(Y_0 \sm \Dl)}$ extends as a singular Hermitian metric 
$g'_{\CO(1)}$ on $\CO(1)|_{\pi^{-1}(Y_0 \sm Z)}$ with semi-positive curvature.
Then by Hartogs type extension,
the singular Hermitian metric $g'_{\CO(1)}$ on $\CO(1)|_{\pi^{-1}(Y_0 \sm Z)}$
extends as a singular Hermitian metric 
$g_{\CO(1)}$ on $\CO(1)|_{\pi^{-1}(Y_0)}$
with semi-positive curvature.

(2)
We can find a closed analytic subset $Z' \subset \Dl$
of $\codim_Y Z' \ge 2$, containing $Z$,
so that we can describe $f : X \sm f^{-1}(Z') \lra Y \sm Z'$
as a union of Set up \ref{local} as in 
Lemma \ref{decomp} without taking any modifications $\mu : X' \lra X$.
Then we obtain the Hodge metric on $\Rq|_{Y \sm \Dl}$
with respect to $\w_f$ and $h$.
The rest of the proof is the same as (1).
\end{proof}





\section{Semi-Stable Reduction} \label{3}

Now our aim is to show Theorem \ref{bdd}.
We shall devote this and next two sections for the proof.
Throughout these three sections, we shall discuss under Set up \ref{local}
and also \S \ref{ssred} below.

\subsection{Weakly semi-stable reduction} \label{ssred}

(\cite[Ch.\ II]{KKMS}\,\cite[\S 7.2]{KM}\,\cite[\S 6.4]{Vi2}.) \
Let 
$$
	f^*\Dl = \sum_j b_j B_j
$$ 
be the prime decomposition.
Let $Y'$ be another copy of a unit polydisc in $\BC^m$ with coordinates
$t' = (t_1', \ldots, t_{m-1}', t_m')$. 
Let $\ell$ be the least common multiple of all $b_j$.
Let $\tau : Y' \lra Y$ be a ramified covering given by
$(t_1', \ldots, t_{m-1}', t_m') \mapsto (t_1', \ldots, t_{m-1}', {t'_m}^\ell)$,
and $\Xo = X \times_Y Y'$ be the fiber product.
Let $\nu : X' \lra \Xo$ be the normalization, and 
$\mu : X'' \lra X'$ be a resolution of singularities,
which is biholomorphic on the smooth locus of $X'$.
\begin{equation*} 
\begin{CD}
    X''  @>{\mu}>>  X' @>{\nu}>> \Xo= X \times_Y Y' @>{\tauo}>> X \\
    @Vf''VV           @Vf'VV           @V{\fo}VV                @VVfV  \\ 
    Y'   @>>{id}>   Y'   @>>{id}>   Y'               @>>{\tau}>   Y
\end{CD}
\end{equation*} 

Then there are natually induced objects:\
$\tauo : \Xo\lra X, \tau' : X' \lra X, \tau'' : X'' \lra X$, 
$\fo : \Xo\lra Y', f' : X' \lra Y', f'' : X'' \lra Y'$,
$E^\circ = {\tauo}^*E, E' = {\tau'}^*E, E'' = {\tau''}^*E$,
and $h''= {\tau''}^*h$ the induced Hermitian metric on $E''$ 
with Nakano semi-positive curvature.
We denote by $j_{\Xo} : \Xo \subset X \times Y'$ the inclusion map,
and by $p_X : X \times Y' \lra X$ and $p_{Y'} : X \times Y' \lra Y'$ 
the projections.
We may also denote by 
$$
	F = \Rq, \ \ \ F' = \FF, 
$$
where $F$ is globally free (Set up \ref{local}), 
and $F'$ is torsion free (Lemma \ref{torsion free}).
Let $\Dl' = \{t'_m = 0\} \subset Y'$.
The discriminant loci of $\fo, f' , f''$ are
contained in $\Dl'$.
We can write 
$$
	{f''}^*\Dl' = \sum_j B''_j + B''_{exc},
$$
where $\sum B''_j$ is the prime decomposition of 
the non-$\mu$-exceptional divisors in ${f''}^*\Dl'$, and 
$B''_{exc}$ is the sum of $\mu$-exceptional divisors in ${f''}^*\Dl'$.
As we will see in Lemma \ref{KM7.23}, 
all coefficients in $\sum B''_j$ are 1.
(As in \cite{KKMS}, ${f''}^*\Dl'$ may be semi-stable in codimension 1.
However we do not need this stronger result for $B''_{exc}$.)

We add a remark on the choice of the smooth model $X''$.
We can assume, possibly after replacing $Y$ by a smaller polydisc,
that $X''$ can be obtained in the following way.
We take an embedded resolution 
$\del : \wtil{X \times Y'} \lra X \times Y'$ of $\Xo$, 
by a finite number of blowing-ups along smooth centers,
which are biholomorphic outside $\Sing \Xo$.
Let us denote by $X'' \subset \wtil{ X \times Y'}$ the smooth model
of $\Xo$, and by $\mu : X'' \lra X'$ the induced morphism.
We may assume further that $\Supp {f''}^*\Dl'$ is simple normal crossing.


\subsection{Direct image sheaves and analysis of singularities}

We will employ algebraic arguments to compair direct image sheaves
on $Y$ and $Y'$, and to study the singularities on $X'$.
We start with an elementary remark.

\begin{lem} \label{locemb}
The normal variety $X'$ is smooth on 
$X' \sm {\tau'}^{-1}(\Sing f^{-1}(\Dl))$,
and the induced map $j_\Xo \circ \nu : X' \lra X \times Y'$ is 
locally embedding around every point on 
$X' \sm {\tau'}^{-1}(\Sing f^{-1}(\Dl))$.
\end{lem}

\begin{proof}
We take a smooth point $x_0$ of $f^{-1}(\Dl)$.
If $x_0 \in B_j$ in $f^*\Dl = \sum b_jB_j$, the map $f$ is given by
$z = (z_1, \ldots, z_{n+m}) \mapsto 
t = (z_{n+1}, \ldots, z_{n+m-1}, z_{n+m}^{b_j})$
for an appropriate local coordinate $(U; z = (z_1, \ldots, z_{n+m}))$
around $x_0$.
Then $\Uo = U \times_Y Y'$ is defined by 
$\Uo = \{(z, t') \in U \times Y';\ f(z) = \tau(t')\}$,
namely $z_{n+1} = t'_1, \ldots, z_{n+m-1} = t'_{m-1}, 
z_{n+m}^{b_j} = {t'_m}^\ell$.
We write $\ell = b_jc_j$ with a positive integer $c_j$.
Let $\ep$ be a $b_j$-th primitive root of unity.
Then $\Uo$ is a union of 
$$
	U^\circ_p = \{(z, t') \in U \times Y';\  z_{n+1} = t'_1, \ldots, 
		z_{n+m-1} = t'_{m-1}, z_{n+m} = \ep^p {t'_m}^{c_j} \}
$$
for $p = 1, \ldots, b_j$.
Each $U^\circ_p$ itself is smooth, and the normalization $U'$ of $\Uo$
is just a disjoint union $\amalg_{p=1}^{b_j} U^\circ_p$.
\end{proof}

The normal variety $X'$ is almost smooth.
For example the following properties are known.

\begin{lem} \label{KM7.23}
\cite[7.23]{KM} (\cite[Ch.\ II]{KKMS}).
(1) The canonical divisor $K_{X'}$ is Cartier,
(2) $X'$ has at most toric, abelian quotient singularities,
(3) a pair $(X', 0)$ is canonical,  and a pair $(X', D')$ is log-canonical, 
where $D' = {f'}^*\Dl'$ which is reduced.
\end{lem}

Since canonical singularities are Cohen-Macaulay,
combined with Lemma \ref{KM7.23}\,(1), we see $X'$ is Gorenstein
(refer \cite[\S 2.3]{KM} including definitions).

\begin{lem} \label{inj}
(cf.\ \cite[Lemma 3.2]{Vi1}\,\cite[V.3.30]{N2}.) \ 
There exists a natural inclusion map
$$
  \vph : F' = \FF \lra  \tau^*F = \tau^* \F,
$$
which is isomorphic over $Y' \sm \Dl'$.
\end{lem}

\begin{proof}
Recall that dualizing sheaves when they exist are flat and compatible 
with any base change \cite[(9)]{Kl}.
The morphism $\nu$ being finite, there exists a dualizing sheaf 
$\w_{X'/\Xo}$ such that 
$\nu_* \w_{X'/\Xo} = \mathcal{H}om_\Xo (\nu_* \CO_{X'}, \CO_\Xo)$.
By base change, $\w_{\Xo/Y'}= {\tauo}^* K_{X/Y}$ is an invertible 
dualizing sheaf for $\fo$. 
Because $Y'$ is smooth, $\w_\Xo = \w_{\Xo/Y'} \ot {\fo}^* K_{Y'}$ 
is an invertible dualizing sheaf for $\Xo$.
In particular $\Xo$ is Gorenstein, 
in fact $\Xo$ is locally complete intersection. 
Now, by composition \cite[(26.vii)]{Kl}, 
$\w_{X'/Y'} = \w_{X'/\Xo} \ot \nu^*\w_{\Xo/Y'}$ is a dualizing sheaf for $f'$.
Because $\w_{\Xo/Y'}$ is locally free, the projection formula reads
$\nu_* \w_{X'/Y'}
= (\nu_* \w_{X'/\Xo}) \ot \w_{\Xo/Y'}
= \mathcal{H}om_\Xo (\nu_* \CO_{X'}, \CO_\Xo \ot \w_{\Xo/Y'})
= \mathcal{H}om_\Xo (\nu_* \CO_{X'}, \w_{\Xo/Y'})$.
Then we have a natural homomorphism $\ga : \nu_*\w_{X'/Y'} \lra \w_{\Xo/Y'}$.
Since $X'$ is Gorenstein and canonical (Lemma \ref{KM7.23}),
we have $K_{X''} = \mu^*K_{X'} + C$
for an effective $\mu$-exceptional divisor $C$, and hence
$\nu_*\mu_*K_{X''/Y'} = \nu_*\mu_*(\mu^*\w_{X'/Y'} \ot \CO_{X''}(C)) 
= \nu_* \w_{X'/Y'}$.
Then the map $\ga$ induces 
$\nu_*\mu_*K_{X''/Y'} \lra \w_{\Xo/Y'} = {\tauo}^*K_{X/Y}$.
We apply $R^qf^\circ_*$ to obtain a map
$R^qf^\circ_*(\nu_*\mu_*(K_{X''/Y'} \ot E'')) 
\lra R^qf^\circ_*({\tauo}^* (K_{X/Y} \ot E))$.

Since $\nu \circ \mu : X'' \lra \Xo$ is birational, we have 
$R^q(\nu \circ \mu)_*K_{X''} = 0$ for $q > 0$ (\cite[6.9]{Tk}).
Noting $E'' = (\nu \circ \mu)^* E^\circ$, we have
$R^qf''_*(K_{X''} \ot E'') 
= R^qf^\circ_*(R^0(\nu \circ \mu)_*(K_{X''} \ot E''))$.
This gives $R^qf^\circ_*(\nu_*\mu_*(K_{X''/Y'} \ot E'')) = \FF$.
On the other hand, since $\tau$ is flat, the base change map
$\tau^* \F \lra R^qf^\circ_*({\tauo}^* (K_{X/Y} \ot E))$ is isomorphic.
Thus we obtain a sheaf homomorphism
$$
  \vph : \FF \lra \tau^* \F.
$$
It is not difficult to see $\vph$ is isomorphic over $Y' \sm \Dl'$,
and hence the kernel of $\vph$ is a torsion sheaf on $Y'$.
The injectivity of $\vph$ is then a consequence of the torsion
freeness of $\FF$, by Lemma \ref{torsion free}.
\end{proof}

As we saw in Lemma \ref{KM7.23}, the singularities of $X'$ are mild.
However we need informations not only on the canonical sheaf of $X'$,
but also on the sheaf of holomorphic $p$-forms on $X'$. 
There are two canonical choices of the definition on a normal variety.
Fortunately both of them coincide for our $X'$.
In the rest of this subsection, $p$ denotes a non-negative integer.

\begin{dfn}
For every $p$, we define the sheaf of holomorphic $p$-forms on $X'$ by 
$\Om_{X'}^p := j_* \Om_{X'_{reg}}^p$,
where $j : X'_{reg} \lra X'$ is the open immersion of the regular part.
\end{dfn}

\begin{lem} \label{toric} 
\cite[1.6]{Da}\,\cite[1.11]{S}.  
$\mu_*\Om_{X''}^p = \Om_{X'}^p$ holds.
\end{lem}

Due to \cite[1.6]{Da}, this lemma is valid not only for our 
$X'$ and $X''$ here, but also more general toric variety $X'$ and 
any resolution of singularities $\mu : X'' \lra X'$.
Our $X'$ is not an algebraic variety, however at every point $x' \in X'$,
there exists an affine toric variety $Z$ with a point $0$
such that $(X', x') \cong (Z, 0)$ as germs of complex spaces.
Hence this lemma follows from \cite[1.6]{Da}.
This is also implicitly contained in the proof of \cite[Lemma 3.9]{Ishida}.

Another key property which we will use, due to Danilov, is the following

\begin{lem} \label{CM}
The sheaf $\Om_{X'}^p$ is Cohen-Macaulay (CM for short),
i.e., at each point $x' \in X'$, the stalk $\Om_{X', x'}^p$ is 
CM as a module over a noetherian local ring 
$(\CO_{X', x'}, \mathfrak m_{X', x'})$.
\end{lem}

\begin{proof}
Let $x' \in X'$.
Since $X'$ has a toric singularity at $x'$,
there exists an affine toric variety $Z$ with a point $0$
such that $(X', x') \cong (Z, 0)$ as germs of complex spaces.
Let $\sg$ be a cone in a finite dimensional vector space $N_\BR$ 
corresponding $Z$ (or a fan $F$ in $N_\BR$ corresponding $Z$).
Since $(X', x') \cong (Z, 0)$ is an abelian quotient singularity
(Lemma \ref{KM7.23}), the cone $\sg$ is simplicial (\cite[3.7]{Dais}).
Then by a result of Danilov (\cite[3.10]{Oda}), 
$\Om_{X'}^p$is CM.
%
\end{proof}

\begin{cor} \label{CMc}
Let $y' \in \Dl'$, and let $(t'_1, \ldots, t'_m)$ be 
(other) coordinates of $Y'$ centered at at $y'$ such that 
$\Dl' = \{t'_m = 0\}$.
Then the central fiber $X'_{y'} \subset X'$ defined by
${f'}^*t'_1 = \cdots =  {f'}^*t'_m = 0$ as a complex subspace 
is pure $n$-dimensional and reduced (\cite[7.23\,(1)]{KM}).
Let $x' \in X'_{y'}$.
Let $s_{m+1}, \ldots, s_{m+n} \in 
\mathfrak m_{X', x'} \subset \CO_{X', x'}$
be a sequence of holomorphic functions such that
$\dim_{x'} (X'_{y'} \cap \{s_{m+1} = \cdots = s_{m+k} = 0\})
= n-k$ for any $1 \le k \le n$.
Then ${f'}^*t'_1, \ldots, {f'}^*t'_m, s_{m+1}, \ldots, s_{m+n}$
is an $\Om_{X', x'}^p$-regular sequence.
\end{cor}

\begin{proof}
Since we already know that $\Om_{X', x'}^p$ is CM,
it is enough to check that
$$
\dim_{x'} \Supp 
\left(\Om_{X', x'}^p / 
({f'}^*t'_1, \ldots, {f'}^*t'_m, s_{m+1}, \ldots, s_{m+n})
\Om_{X', x'}^p
\right) = 0.
$$
(cf.\ \cite[5.1 (1) iff (2)]{KM}\,\cite[III.4.3]{AK}.) \
This is clear by our choice of $s_{m+1}, \ldots, s_{m+n}$.
\end{proof}

\subsection{Non-vanishing}

Recall ${f''}^*\Dl' = \sum B''_j + B''_{exc}$,
where $\sum B''_j$ is the prime decomposition of the non-$\mu$-exceptional
divisors in ${f''}^*\Dl'$, and 
$B''_{exc}$ is the sum of $\mu$-exceptional divisors.

\begin{lem} \label{non-v}
Let $v \in H^0(X'', \Om_{X''}^{n+m-q} \ot E''))$.
Let $y' \in \Dl'$ such that $\Supp {f''}^*\Dl' \lra \Dl'$ is 
relative normal crossing around $y'$.
Assume that $v$ does not vanish at $y'$ as an element of an 
$H^0(Y', \CO_{Y'})$-module, i.e., $f_*''v$ is non-zero in 
$f''_*(\Om_{X''}^{n+m-q} \ot E'')
 /(\mathfrak m_{Y',y'} f''_*(\Om_{X''}^{n+m-q} \ot E''))$.
Then there exists a non-$\mu$-exceptional component $B''_j$ in ${f''}^*\Dl'$ 
such that $v$ does not vanish identically along 
$B''_j \cap {f''}^{-1}(y')$.
\end{lem}

\begin{proof}
Let us denote by $p = n+m-q$.
We have $\mu_*v \in H^0(X', (\mu_*\Om_{X''}^p) \ot E')$.
Recalling Lemma \ref{toric} that $\mu_*\Om_{X''}^p = \Om_{X'}^p$,
we then have
$f''_*v \in H^0(Y', f''_*(\Om_{X''}^p \ot E''))
= H^0(Y', f'_* (\Om_{X'}^p \ot E'))$.
Assume on the contrary that $v$ does vanish identically along 
$B''_j \cap {f''}^{-1}(y')$ for all $j$.
Then it is enough to show that 
$\mu_*v \in H^0(X', {f'}^{-1}\mathfrak m_{Y', y'} \cdot (\Om_{X'}^p \ot E'))$.
In fact it implies that $f'_*(\mu_* v)$ vanishes at $y'$,
and gives a contradiction to that $f''_*v = f'_*(\mu_*v) \in 
H^0(Y', f''_*(\Om_{X''}^p \ot E''))$ does not vanish at $y'$.
Let 
$$
\ga := (\mu_*v)|_{X'_{y'}} 
  \in H^0(X'_{y'}, \left(\Om_{X'}^p / 
	({f'}^*t'_1, \ldots, {f'}^*t'_m)\Om_{X'}^p \right) \ot E').
$$
Then, $\ga = 0$ leads to a contradiction as we want.

We would like to show that the support of $\ga$ is empty. 
Assume on the contrary that there is a point $x' \in X'_{y'}$ 
such that $d := \dim_{x'} \Supp \ga \ge 0$.
Noting that $\mu : X'' \lra X'$ is isomorphic around every point
on $\Reg X'_{y'}$, we see $\Supp \ga \subset \Sing X'_{y'}$, 
because of our assumption that $v$ vanishes identically along 
$B''_j \cap {f''}^{-1}(y')$ for all $j$.
In particular $d < n$.
We take general $s_{m+1}, \ldots, s_{m+n} \in 
\mathfrak m_{X', x'} \subset \CO_{X', x'}$ such that
$\dim_{x'} (X'_{y'} \cap \{s_{m+1} = \cdots = s_{m+k} = 0\})
= n-k$ for any $1 \le k \le n$, and 
$\dim_{x'} (\Supp \ga \cap \{s_{m+1} = \cdots = s_{m+k} = 0\})
= d-k$ for any $1 \le k \le d$.
By the CM property of $\Om_{X', x'}^p$:\ Corollary \ref{CMc},
${f'}^*t'_1, \ldots, {f'}^*t'_m, s_{m+1}, \ldots, s_{m+n}$
form an $\Om_{X', x'}^p \ot E'$-regular sequence.

Assume $d \ge 1$.
We set $\Sigma_d := X'_{y'} \cap \{s_{m+1} = \cdots = s_{m+d} = 0\}$
around $x'$ on which $s_{m+1}, \ldots, s_{m+n}$ are defined,
and consider $\ga|_{\Sigma_d} \in H^0(\Sigma_d,
(\Om_{X'}^p / ({f'}^*t'_1, \ldots, {f'}^*t'_m$, $
	s_{m+1}, \ldots$, $s_{m+d}) \Om_{X'}^p) \ot E')$.
Then $\Supp(\ga|_{\Sigma_d})$ is contained in the zero locus of 
the function $s_{m+d+1}$ around $x'$.
Since $\ga|_{\Sigma_d}$ is non-zero, (some power of $s_{m+d+1}$ and hence) 
$s_{m+d+1}$ is a zero divisor for 
$(\Om_{X', x'}^p / ({f'}^*t'_1, \ldots, {f'}^*t'_m$, 
	$s_{m+1}$, $\ldots, s_{m+d}) \Om_{X', x'}^p) \ot E'_{x'}$,
see \cite[\S 2.2]{GR} R\"uckert Nullstellensatz, cf.\ \cite[II.Ex.5.6]{Ha}.
This gives a contradiction to the fact that
${f'}^*t'_1, \ldots, {f'}^*t'_m$, $s_{m+1}$, $\ldots$, $s_{m+d+1}$
is an $\Om_{X', x'}^p \ot E'_{x'}$-regular sequence.

We also obtain a contradiction assuming $d= 0$,
by a similar manner as above without cutting out by $s_{m+1}$ and so on.
\end{proof}


\section{Hodge Metric on the Ramified Cover}  \label{4}

We still discuss in Set up \ref{local} and \S \ref{ssred}.
To compare the Hodge metric $g$ of $F = \F$ on $Y \sm \Dl$ and
a Hodge metric of $F' = \FF$ on $Y' \sm \Dl'$, 
we need to put an appropriate metric on $X''$.
We can not take arbitrary K\"ahler metric on $X''$ of course.
The problem is that the pull-back ${\tau''}^*\w$ on $X''$ is not 
positive definite any more.
To overcome this problem, we introduce a modified degenerate K\"ahler metric
and a sequence of auxiliary K\"ahler metrics.

\subsection{Degenerate K\"ahler metric}


We consider a direct sum
$$
	\tw := p_X^*\w + p_{Y'}^*\ai \sum dt'_j \wed d\ol{t'_j},
$$
which is a K\"ahler form on $X \times Y'$.
Via the map $j_{\Xo} \circ \nu \circ \mu : X'' \lra X \times Y'$, we let
$$
	\w'' := (j_{\Xo} \circ \nu \circ \mu)^* \tw = (\del^*\tw)|_{X''}
	= {\tau''}^*\w + {f''}^*\ai \sum dt'_j \wed d\ol{t'_j}
$$ 
be a $d$-closed semi-positive $(1,1)$-form on $X''$,
which we may call a degenerate K\"ahler form.
We do not take ${\tau''}^*\w$ as a degenerate K\"ahler form on $X''$,
because it may degenerate totally along ${f''}^{-1}(\Dl')$.
While it is not the case for $\w''$, as we see in the next lemma.
We will denote by $\Exc \mu \subset X''$
the exceptional locus of the map $\mu$.

\begin{lem} \label{positive}
There exists a closed analytic subset 
$V'' \subset X''$ of $\codim_{X''} V'' \ge 2$ and $f''(V'') \subset \Dl'$
such that $\w''$ is a K\"ahler form on $X'' \sm (V'' \cup \Exc \mu)$.
\end{lem}

\begin{proof}
We look at $V' = {\tau'}^{-1}(\Sing f^{-1}(\Dl))$ first, 
which is a closed analytic subset of $X'$ of $\codim_{X'} V' \ge 2$ 
with $f'(V') \subset \Dl'$ and $V' \supset \Sing X'$ by Lemma \ref{locemb}.
We can see that $(j_\Xo \circ (\nu|_{X' \sm V'}))^*\tw$ is
positive definite (i.e., a K\"ahler form) on $X' \sm V'$ as follows.
We continue the argument in the proof of Lemma \ref{locemb},
and use the notations there.
On each $U^\circ_p$ in $U' = \amalg_{p=1}^{b_j} U^\circ_p \subset X'$,
the $(1,1)$-from $(j_\Xo \circ \nu|_{U'})^* \tw$ is $\tw|_{U^\circ_p}$, 
and needless to say it is K\"ahler.
Then our assertion follows from this observation, because
we can write $\mu^{-1}(V') \cup \Exc \mu = V'' \cup \Exc \mu$
for some $V'' \subset X''$ as in the statement.
\end{proof}

The replacement of ${\tau''}^*\w$ by $\w''$ may cause troubles
when we compair Hodge metrics on $Y \sm \Dl$ and $Y' \sm \Dl'$.
However it is not the case by the following isometric lemma.

\begin{lem} \label{isometry}
Let $t \in Y \sm \Dl$ and take one $t' \in Y' \sm \Dl'$ such that 
$\tau(t') = t$, and let $\vph_{t'} : F'_{t'} \lra (\tau^*F)_{t'} = F_t$
be the isomorphism of fibers in Lemma \ref{inj}.
The fiber $F_t$ (resp.\ $F'_{t'}$) has a Hermitian inner product:\
the Hodge metric $g = g_\w$ with respect to $\w$ and $h$ 
(resp.\ $g' = g_{\w''}$ with respect to $\w''$ and $h''$).
Then $\vph_{t'}$ is an isometry with respect to these inner products.
\end{lem}

\begin{proof}
We take a small coordinate neighbourhood $W$ (resp.\ $W'$) around $t$ 
(resp.\ $t'$) such that $\tau|_{W'} : W' \lra W$ is isomorphic, and that
$f'' : X''_{W'} \lra W'$ and $f : X_W \lra W$ are isomorphic
as fiber spaces over the identification $\tau|_{W'} : W' \isom W$.
The Hermitian vector bundle $(E'', h'')$ is ${\tau''}^*(E, h)$
by definition.
If we put a Hermitian inner product $g_{{\tau''}^*\w}$ on $F_{t'}$ 
with respect to ${\tau''}^*\w$ and $h''$, 
the map $\vph_{t'} : (F'_{t'}, g_{{\tau''}^*\w}) \lra (F_t, g_\w)$
is an isometry.
Although $\w'' \ne {\tau''}^*\w$, 
$\w''$ and ${\tau''}^*\w$ are the same as relative K\"ahler forms over $W'$, 
more concretely $\w'' = {\tau''}^*\w + {f''}^*\ai \sum dt_j \wed d\ol{t'_j}$.
Then we have $g_{\w''} = g_{{\tau''}^*\w}$,
by a part of the definition of Hodge metrics \cite[5.2]{MT2}.
\end{proof}

\begin{dfn}
Let $g'$ be the Hodge metric on $\FF|_{Y' \sm \Dl'}$ with respect to
$\w''$ and $h''$.
\end{dfn}


\subsection{Hodge metric with respect to 
 		the degenerate K\"ahler metric}

We would like to develope Takegoshi's theory of ``relative harmonic forms''
with respect to the degenerate K\"ahler form $\w''$ on $X''$.
The goal is the following

\begin{prop} \label{Takegoshi}
(cf.\ \cite[5.2]{Tk}.)\  
There exist $H^0(Y', \CO_{Y'})$-module homomorphisms
\begin{equation*} 
\begin{aligned}
	*\CH &: H^0(Y', \FF) 
		\lra H^0(X'', \Om_{X''}^{n+m-q} \ot E''), \\
	L^q &: H^0(X'', \Om_{X''}^{n+m-q} \ot E'') 
		\lra H^0(Y', \FF) 
\end{aligned}
\end{equation*} 
such that
(1) $(c_{n+m-q}/q!) L^q \circ *\CH = id$, and
(2)
for every $u \in H^0(Y', \FF)$, there exists a relative holomorphic
form $[\sg_u] \in 
H^0(X'' \sm {f''}^{-1}(\Dl'), \Om_{X''/Y'}^{n-q} \ot E'')$ such that
$$
	(* \CH (u))|_{X'' \sm {f''}^{-1}(\Dl')} = \sg_u \wed {f''}^*dt'.
$$ 
\end{prop}

\begin{proof}
{\it Step 1:\ a sequence of K\"ahler forms}.
We take $\{W'_k;\ k = 1, 2, \ldots\}$ a fundamental system of 
neighbourhoods of $\Dl'$ in $Y'$,
such as $W'_k = \{ t' \in Y';\ |t_m'| < 1/(k+1) \}$.
Let $k$ be a positive integer.
Since $\del : \wtil{X \times Y'} \lra X \times Y'$ in \S \ref{ssred} 
is a composition of blowing-ups along smooth centers laying over $\Sing \Xo$,
there exists a $d$-closed real $(1,1)$-form $\xi_k$ on $\wtil{X \times Y'}$
with $\Supp \xi_k \subset (p_{Y'} \circ \del)^{-1}(W'_k)$ such that 
$c_k \del^*\wtil \w + \xi_k > 0$ on $\wtil{X \times Y'}$
for a large constant $c_k$
(possibly after shrinking $Y$ and $Y'$). 
Possibly after replacing $c_k$ by a larger constant, we may assume
$\| \xi_k \|_\infty / c_k \to 0$ as $k \to \infty$.
Here $\| \xi_k \|_\infty$ is the sup-norm with respect to any fixed 
Hermitian metric on $\wtil{X \times Y'}$ 
(possibly after shrinking $Y$ and $Y'$). 
Thus we obtain a sequence of K\"ahler forms
$$
	\{ \wtil \w_k := \del^*\wtil \w + c_k^{-1}\xi_k \}_k
$$ 
on $\wtil{X \times Y'}$ such that $\wtil \w_k = \del^*\wtil \w$
on $\wtil{X \times Y'} \sm (p_{Y'} \circ \del)^{-1}W'_k$,
and $\wtil \w_k \to \del^*\wtil \w$ uniformly on $\wtil{X \times Y'}$ 
as $k \to \infty$.
For every positive integer $k$, we let 
$$
	\w''_k := \tw_k|_{X''}
$$
be a K\"ahler form on $X''$.

{\it Step 2:\ Relative hard Lefschetz type theorem}.
We first recall Takegoshi's theory 
with respect to the K\"ahler forms $\w''_k$ on $X''$. 
Let $W' \subset Y'$ be a Stein subdomain with a strictly plurisubhamonic
exhaustion function $\psi$.
We take a global frame $dt' = dt'_1 \wed \ldots \wed dt'_m$ of $K_{Y'}$.
Recalling $\FF 
= K_{Y'}^{\ot (-1)} \ot R^qf''_*(K_{X''} \ot E'')$,
this trivialization of $K_{Y'}$ gives an isomorphism
$\FF \cong R^qf''_*(K_{X''} \ot E'')$ on $Y'$.
Since $W'$ is Stein, we have also a natural isomorphism
$H^0(W', R^qf''_*(K_{X''} \ot E'')) \cong H^q(X''_{W'}, K_{X''} \ot E'')$,
where $X'' = {f''}^{-1}(W')$.
We denote by $\ga^q$ the composed isomorphism
$$
 \ga^q : H^0(W', \FF) 
	\isom H^q(X''_{W'}, K_{X''} \ot E'').
$$

Let $k$ be a positive integer.
With respect to the K\"ahler form $\w''_k$ on $X''$ in Step 1, 
we denote by $*_k$ the Hodge $*$-operator, and by 
$$
	L_k^q : H^0(X''_{W'}, \Om_{X''}^{n+m-q} \ot E'') 
			\lra H^q(X''_{W'}, K_{X''} \ot E'')
$$ 
the Lefschetz homomorphism induced from ${\w''_k}^q \wed \bullet$. 
Also with respect to $\w''_k$ and $h''$, we set
$\CH^{n+m,q}(X''_{W'}, \w''_k, E'', {f''}^*\psi) 
= \{u \in A^{n+m, q}(X''_{W'}, E'') ;\ \rdb u = \vth_{h''} u = 0, \
	e(\rdb ({f''}^*\psi))^* u = 0 \}$
(see \cite[4.3 or 5.2.i]{Tk}).
By \cite[5.2.i]{Tk}, $\CH^{n+m,q}(X''_{W'}, \w''_k, E'', {f''}^*\psi)$ 
represents $H^q(X''_{W'}, K_{X''} \ot E'')$ as an 
$H^0(Y', \CO_{Y'})$-module, and there exists a natural isomorphism
$$
\iota_k : \CH^{n+m,q}(X''_{W'}, \w''_k, E'', {f''}^*\psi) 
		\isom H^q(X''_{W'}, K_{X''} \ot E'')
$$
given by taking the Dolbeault cohomology class.
We have an isomorphism
$$
	\CH_k = \iota_k^{-1} \circ \ga^q :  
  H^0(W', \FF) 
  \isom \CH^{n+m,q}(X''_{W'}, \w''_k, E'', {f''}^*\psi).
$$
Also by \cite[5.2.i]{Tk},
the Hodge $*$-operator gives an injective homomorphism 
$$
   *_k : \CH^{n+m,q}(X''_{W'}, \w''_k, E'', {f''}^*\psi)
	\lra H^0(X''_{W'}, \Om_{X''}^{n+m-q} \ot E''),
$$
and induces a splitting 
$*_k \circ \iota_k^{-1} : H^q(X''_{W'}, K_{X''} \ot E'')
	\lra H^0(X''_{W'}, \Om_{X''}^{n+m-q} \ot E'')$
for the Lefschetz homomorphism $L_k^q$ such that
$(c_{n+m-q}/q!) L_k^q \circ *_k \circ \iota_k^{-1} = id$.
(The homomorphism $\delta^q$ in \cite[5.2.i]{Tk} with respect to 
$\w''_k$ and $h''$ is $*_k \circ \iota_k^{-1}$ times a universal constant.) \
In particular 
$$
	(c_{n+m-q}/q!) ((\ga^q)^{-1} \circ L_k^q) \circ (*_k \circ \CH_k) 
	= id.
$$
All homomorphisms $\ga^q, *_k, L_k^q, \iota_k, \CH_k$ are as 
$H^0(W', \CO_{Y'})$-modules.

Let $u \in H^0(W', \FF)$. 
Then we have $*_k \circ \CH_k(u) \in H^0(X''_{W'}, \Om_{X''}^{n+m-q} \ot E'')$,
and then by \cite[5.2.ii]{Tk}
$$
	*_k \circ \CH_k(u)|_{X''_{W'} \sm {f''}^{-1}(\Dl')} 
	= \sg_k \wed {f''}^*dt'
$$ 
for some 
$[\sg_k] \in H^0(X''_{W'} \sm {f''}^{-1}(\Dl'), \Om_{X''/Y'}^{n-q} \ot E'')$.
It is not difficult to see 
$[\sg_k] \in H^0(X''_{W'} \sm {f''}^{-1}(\Dl'), \Om_{X''/Y'}^{n-q} \ot E'')$
does not depend on the particular choice of a frame $dt'$ of $K_{Y'}$.

{\it Step 3:\ Takegoshi's theory with respect to $\w''$}.
We then consider the theory for $\w''$.
In case a Stein subdomain $W' \subset Y'$ as above
is contained in $Y' \sm \Dl'$, the theory is the same because 
$\w''$ is K\"ahler on $X'' \sm {f''}^{-1}(\Dl')$ (see Lemma \ref{positive}).
Hence we explain, how to avoid the degeneracy of $\w''$ 
along a part of ${f''}^{-1}(\Dl')$.

Let $k_1$ and $k_2$ be any pair of positive integers.
We take any Stein subdomain $W' \subset Y' \sm (W'_{k_1} \cup W'_{k_2})$,
which admits a smooth strictly plurisubharmonic exhaustion function $\psi$.
Due to \cite[5.2.iv]{Tk}, there are two commutative diagrams for
$i = 1, 2$:\
\begin{equation*} 
\begin{CD}
   H^q(X'', K_{X''} \ot E'')  @>{*_{k_i} \circ \iota_{k_i}^{-1}}>>  
				H^0(X'', \Om_{X''}^{n+m-q} \ot E'')  \\
    @VVV            @VVV   \\
   H^q(X''_{W'}, K_{X''} \ot E'')  @>>{*_{k_i} \circ \iota_{k_i}^{-1}}>  
				H^0(X''_{W'}, \Om_{X''}^{n+m-q} \ot E''). 
\end{CD}
\end{equation*} 
Here the vertical arrows are restriction maps.
The bottom horizontal maps depend only on $\w''_{k_i}|_{X''_{W'}}$.
Recall that $\w''_k = \w''$ on $X'' \sm {f''}^{-1}(W'_k)$.
Because of $\w''= \w''_{k_1} = \w''_{k_2}$ on $X''_{W'}$,
the bottom horizontal maps are independent of $k_1$ and $k_2$.

Let us take $u \in H^0(Y', \FF)$. 
Then by the observation above, two holomorphic forms 
$*_{k_1} \circ \CH_{k_1}(u), *_{k_2} \circ \CH_{k_2}(u) \in 
H^0(X'', \Om_{X''}^{n+m-q} \ot E'')$
coincide on an open subset $X''_{W'}$, and hence
$*_{k_1} \circ \CH_{k_1}(u) = *_{k_2} \circ \CH_{k_2}(u)$ on $X''$.
(Note that it may happen that 
$\CH_{k_1}(u) \ne \CH_{k_2}(u)$ around ${f''}^{-1}(\Dl')$,
because $\CH_k(u) = (c_{n+m-q}/q!) \w''_k \wed (*_k \circ \CH_{k}(u))$ and
$\w''_{k_1} \ne \w''_{k_2}$ around there.) \ 
We denote by 
$$
	* \CH (u) \in H^0(X'', \Om_{X''}^{n+m-q} \ot E'')
$$
instead of arbitrary $*_k \circ \CH_k(u)$. 
Since $\w''$ may not be positive definite along a part of 
${f''}^{-1}(\Dl')$, the operators $*$ and $\CH$ with respect to 
$\w''$ may not be defined across ${f''}^{-1}(\Dl')$.
However 
$$
	*\CH : H^0(Y', \FF) 
		\lra H^0(X'', \Om_{X''}^{n+m-q} \ot E'')
$$
is defined.
Recalling $\CH_k(u) = (c_{n+m-q}/q!) \w''_k \wed (*_k \circ \CH_{k}(u))$
in $H^q(X'', K_{X''} \ot E'')$, since $\CH_{k_1}(u)$ and $\CH_{k_2}(u)$ 
are in the same Dolbeault cohomology class 
$\ga^q(u) \in H^q(X'', K_{X''} \ot E'')$, we have
$L_{k_1}^q (*_{k_1} \circ \CH_{k_1}(u)) 
= L_{k_2}^q (*_{k_2} \circ \CH_{k_2}(u))$.
We put 
$$
   L^q = (\ga^q)^{-1} \circ L_k^q : 
   H^0(X'', \Om_{X''}^{n+m-q} \ot E'') \lra H^0(Y', \FF)
$$ 
for one arbitrary fixed large $k$.
A different choice of $k$ will give a different $L^q$, 
however the relation
$(c_{n+m-q}/q!) ((\ga^q)^{-1} \circ L_k^q) \circ (*_k \circ \CH_k) = id$
in Step 2 implies our assertion (1).
Recall $(*_k \circ \CH_k(u))|_{X'' \sm {f''}^{-1}(\Dl')} 
= \sg_k \wed {f''}^*dt'$ for some 
$[\sg_k] \in H^0(X'' \sm {f''}^{-1}(\Dl'), \Om_{X''/Y'}^{n-q} \ot E'')$.
Then we see, $[\sg_k]$ is also independent of $k$, and hence
$(* \CH (u))|_{X'' \sm {f''}^{-1}(\Dl')}$ can be written as
$$
	(* \CH (u))|_{X'' \sm {f''}^{-1}(\Dl')} = \sg_u \wed {f''}^*dt'
$$ 
for some $[\sg_u] \in 
H^0(X'' \sm {f''}^{-1}(\Dl'), \Om_{X''/Y'}^{n-q} \ot E'')$.
This is (2).
\end{proof}

\begin{rem} 
(1)
We recall the definition of the Hodge metric $g'$ of 
$\FF|_{Y' \sm \Dl'}$ 
with respect to $\w''$ and $h''$.
We remind that $\w''$ is K\"ahler on $X'' \sm {f''}^{-1}(\Dl')$.
We only mention it for a global section 
$u \in H^0(Y', \FF)$. 
It is given by
$$
 g'(u,u)(t') 
	= \int_{X''_{t'}} (c_{n-q}/q!) 
	  ({\w''}^q \wed \sg_u \wed h'' \ol{\sg_u})|_{X''_{t'}}
$$
at $t' \in Y' \sm \Dl'$.

(2)
This is only a side remark, which we will not use later.
The Hodge metric $g'_k$ of $\FF|_{Y' \sm \Dl'}$ 
with respect to $\w''_k$ and $h''$ is given, 
for $u \in H^0(Y', \FF)$, by 
$$
 g'_k(u,u)(t') 
	= \int_{X''_{t'}} (c_{n-q}/q!) 
	  ({\w''_k}^q \wed \sg_{u_k} \wed h'' \ol{\sg_{u_k}})|_{X''_{t'}}
	= \int_{X''_{t'}} (c_{n-q}/q!) 
	  ({\w''_k}^q \wed \sg_u \wed h'' \ol{\sg_u})|_{X''_{t'}}
$$
at $t' \in Y' \sm \Dl'$.
Since $\w''_k \to \w''$ uniformly as $k \to \infty$,
we have $g'_k(u,u)(t') \to g'(u,u)(t')$ as $k \to \infty$,
for any fixed $t' \in Y' \sm \Dl'$.
\qed
\end{rem}

\subsection{Uniform estimate of Fujita type}  \label{Fujita method}

We will give a key estimate of the singularities of
the Hodge metric $g'$ on $\FF|_{Y' \sm \Dl'}$ with respect to 
$\w''$ and $h''$.
This is the main place where we use the fact that, by weakly semi-stable
reduction, we achieve ${f''}^*\Dl'$ is reduced plus $\mu$-exceptional.

In this subsection we pose the following genericity condition
around a point of $\Dl'$.

\begin{assum} \label{a1}
The map $f' : X'' \lra Y'$, $(E'', h'')$ and $F' = \FF$ satisfy 
the conditions (1)--(2) in Set up \ref{local}.
\end{assum}

We then take a global frame $e'_1, \ldots, e'_r \in H^0(Y', F')$
of $F' \cong \CO_{Y'}^{\oplus r}$.
For a constant vector $s = (s_1, \ldots, s_r) \in \BC^r$, we let
$u_s = \sum_{i=1}^r s_ie'_i \in H^0(Y', F')$.
We denote by 
$S^{2r-1} = \{ s \in \BC^r ; \ |s| = (\sum |s_i|^2)^{1/2} = 1 \}$
the unit sphere.

We note the following two things.
Since $e'_1, \ldots, e'_r$ generate $F'$ over $Y'$,
$u_s$ is nowhere vanishing on $Y'$ as soon as $s \ne 0$,
namely $u_s$ is non-zero in $F'/(\mathfrak m_{Y', y'} F')$
at any $y' \in Y'$.
The map $\BC^r \lra H^0(X'', \Om_{X''}^{n+m-q} \ot E'')$
given by $s \mapsto u_s \mapsto *\CH(u_s) = \sum_{i=1}^r s_i (*\CH(e'_i))$ 
is continuous, with respect to the standard topology of $\BC^r$ and 
the topology of $H^0(X'', \Om_{X''}^{n+m-q} \ot E'')$ of 
uniform convergence on compact sets.

\begin{lem}  \label{Ft11}
(cf.\ \cite[1.11]{Ft}.) \ 
Under Assumption \ref{a1} and notations above, 
let $y' \in \Dl'$ and let $s_0 \in S^{2r-1}$. 
Then there exist a neighbourhood $S(s_0)$ of $s_0$ in $S^{2r-1}$,
a neighbourhood $W'_{y'}$ of $y'$ in $Y'$ and a positive number $N$ such that
$g'(u_s, u_s)(t') \ge N$ for any $s \in S(s_0)$ and any 
$t' \in W'_{y'} \sm \Dl'$.
\end{lem}

\begin{proof} 
(1)
We first claim the following variant of Lemma \ref{non-v}.
Let $u \in H^0(Y', F')$, and assume $u$ does not vanish at $y'$. 
Then there exists a non-$\mu$-exceptional component $B''_j$ in 
${f''}^*\Dl' = \sum B''_j + B''_{exc}$, 
such that $* \CH (u) \in H^0(X'', \Om_{X''}^{n+m-q} \ot E'')$
does not vanish identically along $B''_j \cap {f''}^{-1}(y')$.

In fact, by Proposition \ref{Takegoshi}, 
the image $*\CH (H^0(Y', F'))$ is a direct summand of 
$H^0(Y',$ $f''_*(\Om_{X''}^{n+m-q} \ot E''))$ as an $H^0(Y', \CO_{Y'})$-module.
In particular, $* \CH (u) \in H^0(X'', \Om_{X''}^{n+m-q} \ot E'')$
does not vanish at $y' \in Y'$ as an element of an $H^0(Y', \CO_{Y'})$-module.
Then we apply Lemma \ref{non-v}.

(2)
For our nowhere vanishing $u_{s_0}$, we take a non-$\mu$-exceptional 
component 
$$
	B'' = B''_j
$$ 
in ${f''}^*\Dl'$ such that $*\CH (u_{s_0})$ 
does not vanish identically along $B'' \cap {f''}^{-1}(y')$.
We take a general point $x_0 \in B'' \cap {f''}^{-1}(y')$, and a local
coordinate $(U; z = (z_1, \ldots, z_{n+m}))$ centered at $x_0 \in X''$
such that $f''$ is given by
$t' = f''(z) = (z_{n+1}, \ldots, z_{n+m})$ on $U$.
In particular $({f''}^*\Dl')|_U = B''|_U = \{z_{n+m} = 0\}$.
Over $U$, we may assume that the bundle $E''$ is also trivialized, i.e.,
$E''|_U \cong U \times \BC^{r(E)}$, where $r(E)$ is the rank of $E$.
Using these local trivializations on $U$, we have a constant
$a > 0$ such that
(i) $\w'' \ge a \w_{eu}$ on $U$, 
where $\w_{eu} = \ai/2\sum_{i=1}^{n+m} dz_i \wed d\ol{z_i}$
(recall $\w''$ is positive definite around $x_0$ by 
Lemma \ref{positive}!!), and
(ii) $h'' \ge a \text{Id}$ on $U$ as Hermitian matrixes.
Here we regard $h''|_U(z)$ as a positive definite Hermitian matrix 
at each $z \in U$ in terms of $E''|_U \cong U \times \BC^{r(E)}$, 
and here $\text{Id}$ is the $r(E) \times r(E)$ identity matrix.

(3)
Let $s \in S^{2r-1}$.
By Proposition \ref{Takegoshi}, we can write as
$(* \CH (u_s))|_{X'' \sm {f''}^{-1}(\Dl')} = \sg_s \wed {f''}^*dt'$ for some
$\sg_s \in A^{n-q,0}(X'' \sm {f''}^{-1}(\Dl'), E'')$.
We write
$\sg_s = \sum_{I \in I_{n-q}} \sg_{sI} dz_I + R_s$ on $U \sm B''$.
Here $I_{n-q}$ is the set of all multi-indexes 
$1 \le i_1 < \ldots < i_{n-q} \le n$ of length $n-q$ 
(not including $n+1, \ldots, n+m$),
$\sg_{sI} = {}^t(\sg_{sI,1}, \ldots, \sg_{sI,r(E)})$ is a row vector valued 
holomorphic function with $\sg_{sI,i} \in H^0(U \sm B'', \CO_{X''})$, and here 
$R_s = \sum_{k=1}^m R_{sk} \wed dz_{n+k} \in A^{n-q,0}(U \sm B'', E'')$. 
Then 
$$
 \sg_s \wed {f''}^*dt' 
	 = \bigg(\sum_{I \in I_{n-q}} \sg_{sI} dz_I \bigg) 
		\wed dz_{n+1} \wed \ldots \wed dz_{n+m}
$$ 
on $U \sm B''$.
Since $\sg_s \wed {f''}^*dt' = (* \CH (u_s))|_{X'' \sm {f''}^{-1}(\Dl')}$ 
and $* \CH (u_s) \in H^0(X'', \Om_{X''}^{n+m-q} \ot E'')$,
all $\sg_{sI}$ can be extended holomorphically on $U$.
We still denote by the same latter
$\sg_{sI} = {}^t(\sg_{sI,1}, \ldots, \sg_{sI,r(E)})$ its extension.

At the point $s_0 \in S^{2r-1}$, since $*\CH (u_{s_0})$ does not vanish 
identically along $B'' \cap {f''}^{-1}(y')$, and since 
$x_0 \in B'' \cap {f''}^{-1}(y')$ 
is general, we have at least one $\sg_{s_0J_0,i_0} \in H^0(U, \CO_{X''})$
such that $\sg_{s_0J_0, i_0}(x_0) \ne 0$.
We take such 
$$
	J_0 \in I_{n-q} \text{ and } i_0 \in \{1, \ldots, r(E)\}.
$$

(4)
By the continuity of $s \mapsto u_s \mapsto *\CH(u_s)$, 
we can take an $\ep$-polydisc 
$U(\ep) = \{z = (z_1, \ldots, z_{n+m}) \in U ; \ 
|z_i| < \ep \text{ for any } 1 \le i \le n+m \}$ centered at $x_0$ for some
$\ep > 0$, and a neighbourhood $S(s_0)$ of $s_0$ in $S^{2r-1}$ such that
$$
	A := \inf\{ |\sg_{sJ_0,i_0}(z)| ; \
		 s \in S(s_0), \ z \in U(\ep) \} > 0.
$$
We set $W'_{y'} := f''(U(\ep))$, which is an open neighbourhood of
$y' \in Y'$, since $f''$ is flat (in particular it is open).
Then for any $s \in S(s_0)$ and any $t' \in W'_{y'} \sm \Dl'$, 
we have
\begin{equation*} 
\begin{aligned}
\int_{X''_{t'}} (c_{n-q}/q!) 
	({\w''}^q \wed \sg_{s} \wed h''\ol{\sg_{s}})|_{X''_{t'}} 
&
\ge
a \int_{X''_{t'} \cap U}(c_{n-q}/q!) 
	({\w''}^q \wed \sg_{s} \wed \ol{\sg_{s}})|_{X''_{t'} \cap U} \\
& = 
a^{q+1} \int_{z \in X''_{t'} \cap U}
	\sum_{I \in I_{n-q}} \sum_{i=1}^{r(E)} |\sg_{sI,i}(z)|^2 dV_n \\
& \ge 
a^{q+1} \int_{z \in X''_{t'} \cap U(\ep)} A^2 \ dV_n \\
& = 
a^{q+1} A^2 (\pi \ep^2)^n.
\end{aligned}
\end{equation*} 
Here $dV_n = (\ai/2)^n \bigwedge_{i=1}^n dz_i \wed d\ol{z_i}$
is the standard euclidean volume form on $\BC^n$.
\end{proof}

\begin{lem} \label{Ft12}
(cf.\ \cite[1.12]{Ft}.) \ 
Under Assumption \ref{a1} and notations after that, 
let $y' \in \Dl'$. 
Then there exist a neighbourhood $W'_{y'}$ of $y'$ in $Y'$
and a positive number $N$, such that 
$g'(u_s, u_s)(t') \ge N$ for any $s \in S^{2r-1}$
and any $t' \in W'_{y'} \sm \Dl'$.
\end{lem}

\begin{proof}
Since $S^{2r-1}$ is compact, this is clear from Lemma \ref{Ft11}.
\end{proof}



\section{Plurisubharmonic Extension} \label{5}

We still discuss in Set up \ref{local} and \S \ref{ssred}.
We are ready to talk about, say ``the plurisubharmonic extension''
of the quotient metric $\go$ of $\CO(1)|_{\pi^{-1}(Y \sm \Dl)}$
in Theorem \ref{bdd}.
Since such an extension 
is a local question on $\BP(F)$,
we shall discuss around a fixed point $P \in \BP(F)$.
We take a quotient line bundle $F \lra L$ so that
$P$ corresponds to $F_{\pi(P)} \lra L_{\pi(P)}$.
We also take a trivialization of $F$ given by 
$e_1, \ldots, e_r \in H^0(Y, F)$, so that
the kernel $M$ of $F \lra L$ is generated by $e_1, \ldots, e_{r-1}$.
A choice of a frame $e_1, \ldots, e_r$ also gives a trivialization
$\BP(F) \cong Y \times \BP^{r-1}$.
From now on, we identify $\BP(F)$ and $Y \times \BP^{r-1}$.

\subsection{Quotient metric}

We first describe the quotient metric $\go$ around $P$.
Let $[a] = (a_1 : \ldots : a_r)$ be the homogeneous coordinate of $\BP^{r-1}$.
Then $P = \pi(P) \times (0 : \ldots : 0 : 1)$ in $Y \times \BP^{r-1}$.
Let 
$$
	U = Y \times \{[a] \in \BP^{r-1} ;\ a_r \neq 0 \}
$$
be a standard open neighbourhood of $P$.
This neighbourhood of $P$ (or of $F_{\pi(P)} \lra L_{\pi(P)}$) 
is also described as follows.
Let $a = (a_1, \ldots, a_{r-1}) \in \BC^{r-1}$ 
(be an inhomogeneous coordinate of $\BP^{r-1}$).
We set $e_{ia} = e_i+a_ie_r \in H^0(Y, F)$ for every $1 \le i \le r-1$, 
and $e_{ra} = e_r$,
and let $M_a$ be the subbundle of $F$ generated by $e_{1a}, \ldots, e_{r-1a}$,
and let $L_a = F/M_a$ be the quotient line bundle on $Y$.
Every point $t \times a \in U$ corresponds to a subspace 
$M_{at} \subset F_t$ generated by $e_1(t)+a_1e_r(t), \ldots,
e_{r-1}(t)+a_{r-1}e_r(t)$ and hence the quotient space $L_{at} = F_t/M_{at}$.
For every fixed $a \in \BC^{r-1}$, we have a nowhere vanishing section 
$$
	\what e_{ra} \in H^0(Y, L_a)
$$ 
defined by $\what e_{ra} : t \in Y \mapsto \what e_{ra}(t) \in L_{at}$.
Here $\what e_{ra}(t)$ is the image of $e_{r}(t) \in F_t$
under the quotient $F_t \lra L_{at}$.
We have a canonical nowhere vanishing section
$$
	\what e_r \in H^0(U, \CO(1))
$$ 
defined by 
$\what e_r : t \times a \in U \mapsto \what e_{ra}(t) \in L_{at}$.

Let $a \in \BC^{r-1}$.
With respect to the global frame $\{e_{ia}\}_{i=1}^r$ of $F$, 
the Hodge metric $g$ on $F|_{Y \sm \Dl}$ is written as
$g_{i\ol j a} := g(e_{ia}, e_{ja}) \in A^0(Y \sm \Dl, \BC)$ for 
$1 \le i, j \le r$.
At each point $t \in Y \sm \Dl$, $(g_{i\ol j a}(t))_{1 \le i, j \le r}$
is a positive definite Hermitian matrix, in particular
$(g_{i\ol j a}(t))_{1 \le i, j \le r-1}$ is also positive definite.
We let $(g^{\ol i j}_a(t))_{1 \le i, j \le r-1}$ be the inverse matrix.
The pointwise orthogonal projection of $e_r$ to 
$(M_a|_{Y \sm \Dl})^\perp$ 
with respect to $g$ is given by
$$
  P_a(e_r) 
  = e_r - \sum_{i=1}^{r-1}\sum_{j=1}^{r-1}e_{ia}g^{\ol i j}_ag_{j \ol{r} a}
  \ \in A^0(Y \sm \Dl, F).
$$
Then the quotient metric $g_{L_a}$ on the line bundle ${L_a}|_{Y \sm \Dl}$ 
is described as
$$
	 g_{L_a}(\what e_{ra}, \what e_{ra}) = g(P_a(e_r), P_a(e_r)).
$$
Then, at each 
$t \times a \in U \sm \pi^{-1}(\Dl) = (Y \sm \Dl) \times \BC^{r-1}$, 
we have
$$
	\go(\what e_r, \what e_r)(t \times a) 
	= g_{L_a}(\what e_{ra}, \what e_{ra})(t). 
$$
We already know that 
$-\log(\go(\what e_r, \what e_r)|_{(Y \sm \Dl) \times \BC^{r-1}})$ 
is plurisubharmonic (\cite[1.2]{B}\,\cite[1.1]{MT2}).
What we want to prove is

\begin{lem} \label{bdd2}
Let $\ep$ be a real number such that $0 < \ep < (2(r-1))^{-2}$,
and let $D_\ep = \{ a = (a_1, \ldots, a_{r-1}) \in \BC^{r-1};\
\sum_{i=1}^{r-1}|a_i|^2 < \ep \}$.
Then $-\log(\go(\what e_r, \what e_r)|_{(Y \sm \Dl) \times D_\ep})$ 
extends as a plurisubharmonic function on $Y \times D_\ep$.
\end{lem}

In case $r = 1$, this (as well as Lemma \ref{bdd3} and \ref{bdd4} below) 
should be read that \linebreak
$-\log(\go(\what e_r, \what e_r)|_{Y \sm \Dl}) 
= -\log(g(e_1, e_1)|_{Y \sm \Dl})$
extends as a plurisubharmonic function on $Y$.
Since $P \in \BP(F)$ is arbitrary,
this lemma implies Theorem \ref{bdd}.

\subsection{Boundedness and reduction on the ramified cover}

In Lemma \ref{inj}, we have a natural inclusion $\vph : F' \lra \tau^*F$, 
which is isomorphic over $Y' \sm \Dl'$.
We will reduce our study of $F$ to that of $F'$ via this $\vph$.
Let $L' \subset \tau^*L$ be the image of the composition
$F' \lra \tau^*F \lra \tau^*L$,
and let $M'$ be the kernel of the quotient $F' \lra L'$.
Then we have the following commutative diagram:
\begin{equation*} 
\begin{CD}
    0  @>>>  M'    @>>>      F'   @>>>    L'   @>>> 0 \\
    @.     @VVV    @VV{\vph}V    @VVV     @.  \\ 
    0  @>>> \tau^*M  @>>> \tau^*F @>>> \tau^*L @>>> 0. 
\end{CD}
\end{equation*} 
Here, horizontals are exact, verticals are injective. 
Since $F', L'$ and $M'$ are all torsion free $\CO_{Y'}$-module sheaves, 
we can find a closed analytic subset $Z' \subset \Dl'$
of $\codim_{Y'} Z' \ge 2$ such that
$F', L'$ and $M'$ are all locally free on $Y' \sm Z'$.
We may also assume that $f''$ is flat over $Y' \sm Z'$, and 
$\Supp {f''}^*\Dl' \lra \Dl'$ is relative normal crossing over $\Dl' \sm Z'$.
We set $Z = \tau(Z') \subset \Dl$ a closed analytic subset 
of $\codim_Y(Z) \ge 2$.
We then take an arbitrary point 
$$
	y \in \Dl \sm Z \text{ and let } 
	y' = \tau^{-1}(y) \in \Dl' \sm Z'. 
$$
Then Lemma \ref{bdd2} is reduced to the following

\begin{lem} \label{bdd3}
There exists a neighbourhood $W_y$ of $y$ in $Y$
such that $\go(\what e_{r}, \what e_{r})$ is bounded
from below by a positive constant on $(W_y \sm \Dl) \times D_\ep$,
for $D_\ep$ in Lemma \ref{bdd2}.
\end{lem}

In fact, since $y \in \Dl \sm Z$ is arbitrary, by Riemann type extension, 
$-\log(\go(\what e_{r}, \what e_{r}))$ 
becomes plurisubhamonic on $(Y \sm Z) \times D_\ep$, and then it is
plurisubhamonic on $Y \times D_\ep$ by Hartogs type extension.


To show Lemma \ref{bdd3}, we need to analyze the map 
$\vph : F' \lra \tau^*F$ and its inverse.
We shall formulate and prove a quantitative version of Lemma \ref{bdd3}
as Lemma \ref{bdd4}.

Since our assertion in Lemma \ref{bdd3} is local around the point
$y$ (and $y'$) and over there for $\pi : \BP(F) \lra Y$,
by replacing $Y$ (resp.\ $Y'$) by a small polydisc centered at $y$ 
(resp.\ $y'$), we can also assume that $F' \cong \CO_{Y'}^{\oplus r}$.
In particular the assumption to use Lemma \ref{Ft11} and Lemma \ref{Ft12} 
is satisfied (remind also the choice of $Z'$).
We take a global frame $e'_1, \ldots, e'_r \in H^0(Y', F')$ of $F'$ 
such that $e'_1, \ldots, e'_{r-1}$ generate $M'$ and 
the image $\what e'_r \in H^0(Y', L')$ of $e'_r$
under $F' \lra L'$ generates $L'$.
We still use (the restriction of) the same global frame 
$e_1, \ldots, e_r \in H^0(Y, F)$ of $F$, 
although the point $\pi(P)$ may not belong to the new $Y$ any more.

In terms of those frames $\{\tau^*e_j\}$ and $\{e'_j\}$,
we represent the bundle map $\vph : F' \lra \tau^*F$ on $Y'$.
For each $j$, we write $\vph(e'_j) = \sum_{i=1}^r (\tau^*e_i) \vph_{ij}$
for some $\vph_{ij} \in H^0(Y', \CO_{Y'})$.
Then $\vph$ is given by $\Phi = (\vph_{ij})_{1 \le i,j \le r}$ 
an $r \times r$-matrix valued holomorphic function on $Y'$.
Since $\vph(e'_j)$ for $1 \le j \le r-1$ belongs to $H^0(Y', \tau^*M)$,
we have $\vph_{r1} = \ldots = \vph_{r r-1} \equiv 0$.
We write 
$$
\Phi =
	\begin{pmatrix}
	\Phi_0 & \vph_{*r} \\
	0 \ \ \cdots \ \ 0 & \vph_{rr} 
	\end{pmatrix}, 
$$
accordingly so that 
$(\vph(e'_1), \ldots, \vph(e'_r)) = (\tau^*e_1, \ldots, \tau^*e_r) \Phi$.
Here $\vph_{*r} = {}^t(\vph_{1r}, \ldots, \vph_{r-1 r})$, and
the last part $\vph_{rr}$ represents the line bundle homomorphism
$L' \lra \tau^*L$ on $Y'$.

By replacing $Y$ and $Y'$ by smaller polydiscs, we may assume that 
there exists a constant $C_{\Phi 1} > 0$ such that 
$$
	|\vph_{ij}(t')| < C_{\Phi 1}
$$ 
for any pair $1 \le i, j \le r$ and any $t' \in Y'$.
Since $\vph$ is isomorphic over $Y' \sm \Dl'$, 
we can talk about the inverse there.
Let $\Phi^{-1} = (\vph^{ij})_{1 \le i, j \le r}$
be the inverse on $Y' \sm \Dl'$.
Then $\Phi_0^{-1} = (\vph^{ij})_{1 \le i, j \le r-1}$,
$\vph^{r1} = \ldots = \vph^{r r-1} \equiv 0$,
$\vph^{rr} = \vph_{rr}^{-1}$, and
$\vph^{ir} = - (\sum_{j=1}^{r-1} \vph^{ij} \vph_{jr}) \vph_{rr}^{-1}$:\
$$
\Phi^{-1} =
	\begin{pmatrix}
	\Phi_0^{-1} & - \Phi_0^{-1} \vph_{*r} \vph_{rr}^{-1} \\
	0 \ \ \cdots \ \ 0 & \vph_{rr}^{-1} 
	\end{pmatrix}. 
$$
Needless to say,
$(\tau^*e_1, \ldots, \tau^*e_r) = (\vph(e'_1), \ldots, \vph(e'_r)) \Phi^{-1}$.

\begin{lem} \label{eigen}
Assume $r>1$.
Let $\Psi := \Phi_0 {}^t\ol{\Phi_0} \in A^0(Y', M(r-1, \BC))$ be a 
matrix valued smooth function on $Y'$.
Then there exists a constant $C_{\Phi 2} > 0$ such that
$$
	\lambda_1(\Psi^{-1}(t')) \ge 1/C_{\Phi 2}
$$ 
for any $t' \in Y' \sm \Dl'$, where $\lambda_1(\Psi^{-1}(t'))$ 
is the smallest eigenvalue of the Hermitian matrix $\Psi^{-1}(t')$.
\end{lem}

\begin{proof}
(1)
At each $t' \in Y'$, $\Psi(t')$ is a Hermitian matrix which is semi-positive.
Moreover it is positive definite for any $t' \in Y' \sm \Dl'$,
since $\Phi_0$ is non-singular on it.
All entries of $\Psi$ are also bounded by a constant on $Y'$,
namely if $\Psi = (\psi_{ij})_{1 \le i, j \le r}$ with 
$\psi_{ij} \in A^0(Y', \BC)$, then $|\psi_{ij}(t')| < (r-1)C_{\Phi 1}^2$
for any pair $1 \le i, j \le r$ and any $t' \in Y'$.
In particular, as we will see below (2), 
there exists a constant $C_{\Phi 2} = (r-1)^2C_{\Phi 1}^2 > 0$ such that
$\lambda_{r-1}(\Psi(t')) \le C_{\Phi 2}$ for any $t' \in Y'$,
where $\lambda_{r-1}(\Psi(t'))$ is the biggest eigenvalue of the matrix
$\Psi(t')$.
On $Y' \sm \Dl'$, we have the inverse $\Psi^{-1}$, whose pointwise
matrix value $\Psi^{-1}(t')$ is also positive definite at each 
$t' \in Y' \sm \Dl'$. 
Then $\lambda_1(\Psi^{-1}(t')) = 1/\lambda_{r-1}(\Psi(t')) 
\ge 1/C_{\Phi 2}$
for any $t' \in Y' \sm \Dl'$.

(2)
We consider in general, a non-zero matrix $A = (a_{ij}) \in M(n, \BC)$.
Let $C = \max \{|a_{ij}| \ ; \ 1 \le i, j \le n\}$.
Then we have $|\lambda| \le nC$ for any eigenvalue $\lambda$ of 
$A$ as follows.
Let $v = {}^t(v_1, \ldots, v_n)$ be a non-zero vector
such that $Av = \lambda v$, and take $p$ such that
$|v_p| = \max \{|v_j| \ ; \ 1 \le j \le n \} > 0$.
Then $\lambda v_p = \sum_{j=1}^n a_{pj}v_j$, and
$|\lambda||v_p| \le \sum_{j=1}^n |a_{pj}||v_j| \le nC |v_p|$.
Hence $|\lambda| \le nC$. 
\end{proof}

We set $C_\Phi = \max \{C_{\Phi 1}, C_{\Phi 2}, 1\}$.


\subsection{Final uniform estimate}

The following is a quantitative version of Lemma \ref{bdd3}:\

\begin{lem} \label{bdd4}
Let $y \in \Dl \sm Z$ and $y' \in \Dl' \sm Z'$ as above in Lemma \ref{bdd3}. 
Let $W'_{y'}$ be a neighbourhood of $y'$ and $N$ be a positive number
as in Lemma \ref{Ft12}, and set $W_y = \tau(W'_{y'})$ a neighbourhood of $y$.
Then $g_{L_a}(\what e_{ra}, \what e_{ra})(t) 
\ge N(2C_\Phi)^{-2}$
for any $t \in W_y \sm \Dl$ and any $a \in D_\ep$.
\end{lem}

\begin{proof}
We take arbitrary $t \in W_y \sm \Dl$ and $a \in D_\ep$,
and take one $t' \in W'_{y'}$ such that $\tau(t') = t$.
In case $r=1$, we have
$\go(\what e_r, \what e_r)(t) 
= g(e_1, e_1)(t) = g'(e_1', e_1')(t')|\vph_{11}^{-1}(t')|^2$.
Then by Lemma \ref{Ft12}, 
$g'(e_1', e_1')(t')|\vph_{11}^{-1}(t')|^2 \ge N C_{\Phi 1}^{-2}$. 
This proves Lemma \ref{bdd3} in case $r=1$.
For the rest, we consider in case $r > 1$.

(1)
We reduce an estimate on $g_{L_a}$ to that on $g'$ as follows.
We set $\sg_{ia} = \sum_{j=1}^{r-1} g^{\ol i j}_a g_{j \ol{r} a}$
for $1 \le i \le r-1$ and $\sg_{ra} = 1 - \sum_{i=1}^{r-1} \sg_{ia}a_i$,
which are in $A^0(Y \sm \Dl, \BC)$.
We can write as
$P_a(e_r) = \sg_{ra}e_r - \sum_{i=1}^{r-1} \sg_{ia}e_i$ on $Y \sm \Dl$.
Then
$\tau^* P_a(e_r) 
= \sg_{ra} \vph_{rr}^{-1} \vph(e_r') 
  + \sum_{i=1}^{r-1} (\sg_{ra}\vph^{ir} - \sum_{j=1}^{r-1} \sg_{ja}\vph^{ij}) 
		\vph(e_i')$, and
$$
\vph^{-1}\tau^* P_a(e_r) 
= \sg_{ra} \vph_{rr}^{-1} e_r' 
  + \sum_{i=1}^{r-1} (\sg_{ra}\vph^{ir} - \sum_{j=1}^{r-1} \sg_{ja}\vph^{ij}) 
		e_i'
$$
on $Y' \sm \Dl'$.
Recall $g_{L_a}(\what e_{ra}, \what e_{ra})(t) = g(P_a(e_r), P_a(e_r))(t)$, 
and $g(P_a(e_r), P_a(e_r))(t)
	= g'(\vph_{t'}^{-1}\tau^* P_a(e_r), 
				\vph_{t'}^{-1}\tau^* P_a(e_r))(t')$
by Lemma \ref{isometry}. 
We set $s_r = \sg_{ra}(t') \vph_{rr}^{-1}(t')$ and
$s_i = \sg_{ra}(t')\vph^{ir}(t') - \sum_{j=1}^{r-1} \sg_{ja}(t')\vph^{ij}(t')$
for $1 \le i \le r-1$. 
We obtain a non-zero vector $s = (s_1, \ldots, s_r) \in \BC^r$.
Then
$\vph_{t'}^{-1}\tau^* P_a(e_r) = u_s(t') = \sum_{i=1}^{r} s_i e'_i(t')$
at the $t'$.
Hence it is enough to show $g'(u_s, u_s)(t') \ge N(2C_\Phi)^{-2}$.

(2)
We claim that $|s|^2 := \sum_{i=1}^r |s_i|^2 \ge (2C_\Phi)^{-2}$.
This claim, combined with Lemma \ref{Ft12}, implies that
$g'(u_s, u_s)(t') = |s|^2 g'(u_{s/|s|}, u_{s/|s|})(t') 
\ge |s|^2 N \ge N(2C_\Phi)^{-2}$.

(3)
We prove the claim in (2).
By using the formula on $\Phi^{-1}$, we have
$$
	s_i = - \sum_{j=1}^{r-1} 
		\left\{ \sg_{ra}(t')\vph_{rr}^{-1}(t')\vph_{jr}(t') 
			+ \sg_{ja}(t') \right\} \vph^{ij}(t')
$$
for $1 \le i \le r-1$. 
We set $v_j = \sg_{ra}(t')\vph_{rr}^{-1}(t')\vph_{jr}(t') + \sg_{ja}(t')$
for $1 \le j \le r-1$. 
Then ${}^t (s_1, \ldots, s_{r-1}) 
= - \Phi_0^{-1}(t') \cdot {}^t(v_1, \ldots, v_{r-1})$, and
$\sum_{i=1}^{r-1}|s_i|^2 = \langle \Phi_0^{-1}(t') v, \Phi_0^{-1}(t') v \rangle
= \langle \Psi(t')^{-1}v, v \rangle$.
Here $v = {}^t(v_1, \ldots, v_{r-1})$, and the bracket $\langle \ \ \rangle$ 
is the standard Hermitian inner product on $\BC^{r-1}$,
and recall $\Psi = \Phi_0 {}^t\ol{\Phi_0}$.
Then
$\langle \Psi(t')^{-1}v, v \rangle \ge \sum_{i=1}^{r-1}|v_i|^2/C_\Phi$
by Lemma \ref{eigen}.

In case $|s_r| \ge (2C_\Phi)^{-1}$, our claim in (2) is clear.
Hence we assume $|s_r| < (2C_\Phi)^{-1}$, namely
$|\sg_{ra}(t')||\vph_{rr}^{-1}(t')| < (2C_\Phi)^{-1}$.
Then $|1 - \sum_{i=1}^{r-1} \sg_{ia}(t')a_i| = 
|\sg_{ra}(t')| < |\vph_{rr}(t')| (2C_\Phi)^{-1}$ $< 1/2$. 
We have at least one $1 \le j \le r-1$ such that 
$|\sg_{ja}(t')||a_j| > 1/(2(r-1))$.
In particular $|\sg_{ja}(t')|> 1/(2(r-1) |a_j|) > 1/(2(r-1) \sqrt{\ep})$.
Then for such $j$,
$|v_j| = |\sg_{ja}(t') + \sg_{ra}(t')\vph_{rr}^{-1}(t')\vph_{jr}(t')|
\ge |\sg_{ja}(t')| - |s_r||\vph_{jr}(t')|
> 1/(2(r-1) \sqrt{\ep}) - (2C_\Phi)^{-1} C_\Phi$.
Using $\ep < (2(r-1))^{-2}$, we have $v_j > 1/\sqrt{2}$.
Then we have $\sum_{i=1}^{r-1}|s_i|^2 
> \sum_{i=1}^{r-1}|v_i|^2/C_\Phi > (2C_\Phi)^{-1}$, 
and hence our claim in (2).
\end{proof}

Thus we have proved all Lemma \ref{bdd3}, Lemma \ref{bdd2},
and hence Theorem \ref{bdd}.


\section{Proof of Theorem {\ref{algebraic}} and Variants} \label{6}

\subsection{Proof of Theorem \ref{algebraic}}

The projectivity assumption on $Y$ is only used to define
the weakly positivity of sheaves.
As we will see in the proof below, it is enough to assume that
$f : X \lra Y$ is a K\"ahler fiber space over 
a smooth projective variety $Y$.

After obtaining Theorem \ref{MT}, the proof is standard and classical.
A minor difficulty in analytic approach will be that the sheaf
$\Rq$ may not be locally free in general.

Let $F$ be, in general, a torsion free coherent sheaf on a smooth
projective variety $Y$, and let $Y_1$ be the maximum Zariski open
subset of $Y$ on which $F$ is locally free.
Let $Y_0$ be a Zariski open subset of $Y$, which is contained in $Y_1$.
The sheaf $F$ is said to be {\it weakly positive} over $Y_0$
in the sense of Viehweg \cite[2.13]{Vi2},
if for any given ample line bundle $A$ on $Y$ and any given positive 
integer $a$, there exists a positive integer $b$ such that
$\what S^{ab}(F) \ot A^{\ot b}$ is generated by global sections 
$H^0(Y, \what S^{ab}(F) \ot A^{\ot b})$ over $Y_0$.
Here $\what S^m(F)$ is the double dual of the $m$-th symmetric 
tensor product $\text{Sym}^m(F)$.
We note \cite[2.14]{Vi2} that this condition does not depend on 
the choice of $A$.
We refer also \cite[V.3.20]{N2}.

Now we turn to our situation in Theorem \ref{algebraic}.
Let us denote by $F = \Rq$ which is a torsion free sheaf on $Y$.
Then by \cite[2.14]{Vi2}, it is enough to show that 
there exists an ample line bundle $A$ on $Y$ 
with the following property:\
for any positive integer $a$, there exists a positive integer $b$ 
such that $\what S^{ab}(F) \ot A^{\ot b}$
is generated by $H^0(Y, \what S^{ab}(F) \ot A^{\ot b})$ 
over $Y \sm \Dl$.

Associated to $F$ on $Y$, we have a scheme 
$\BP(F) = \text{Proj}(\bigoplus_{m \ge 0} \text{Sym}^m(F))$
over $Y$, say $\pi : \BP(F) \lra Y$, and 
the tautological line bundle $\CO(1)$ on $\BP(F)$.
Let $\BP'(F) \lra \BP(F)$ be the normalization of the component
of $\BP(F)$ containing $\pi^{-1}(Y \sm S_q)$, and let
$Z' \lra \BP'(F)$ be a birational morphism from a smooth projective
variety that is an isomorphism over $Y \sm S_q$ (\cite[V.\S 3.c]{N2}).
In particular $\BP(F) \sm \pi^{-1}(S_q)$ is a Zariski open subset
of a smooth projective variety $Z'$, in particular 
it admits a complete K\"ahler metric \cite[0.2]{D82}.
We denote by $Z = \BP(F) \sm \pi^{-1}(S_q)$, and take 
a complete K\"ahler form $\w_Z$ on $Z$.
The volume form will be denoted by $dV$.

We take a very ample line bundle $A$ on $Y$
such that $A \ot K_Y^{-1} \ot (\what \det F)^{-1}$ is ample,
where $\what \det F$ is the double dual of $\bigwedge^r F$
and $r$ is the rank of $F$.
Let $h_{K_Y}$ (resp.\ $h_{\what \det F}$) be a smooth 
Hermitian metric on $K_Y$ (resp.\ $\what \det F$),
and let $h_A$ be a smooth Hermitian metric on $A$
with positive curvature, and such that
$h_A h_{K_Y}^{-1} h_{\what \det F}^{-1}$ has positive curvature too.
Let $a$ be a positive integer.
Then, noting that $\what S^{ab}(F) \ot A^{\ot b}$ is reflexive,
it is enough to show that the restriction map
$$
  H^0(\BP(F) \sm \pi^{-1}(S_q), \CO(ab) \ot \pi^*A^{\ot b})
  \lra 
  H^0(\BP(F_y), (\CO(ab) \ot \pi^*A^{\ot b})|_{\BP(F_y)})
$$
is surjective for any $y \in Y \sm \Dl$ and any integer $b > m+1$,
where $m = \dim Y$.
We now fix $y \in Y \sm \Dl$ and $b > m+1$.

We take general members $s_1$, $\ldots$, $s_m \in H^0(Y, A)$ 
such that the zero divisors $(s_1)_0$, $\ldots$, $(s_m)_0$ are smooth,
and intersect transversally, and such that $y$ is isolated in 
$\bigcap_{i=1}^m (s_i)_0$.
Let $W_y \subset Y \sm \Dl$ be an open neighbourhood of $y$,
which is biholomorphic to a ball in $\BC^m$ of radius 2,
$W_y \cap \bigcap_{i=1}^m (s_i)_0 = \{y\}$, and
$F|_{W_y}$ is trivialized.
Let $\rho \in A^0(Y, \BR)$ be a cut-off function around $y$
such that $0 \le \rho \le 1$ on $W_y$, 
$\Supp \rho \subset W_y$, and $\rho \equiv 1$
on $W_y'$ the ball of radius 1 in $W_y$.
Let $\phi = \log (\sum_{i=1}^m h_A(s_i, s_i))^m \in L^1_{loc}(Y, \BR)$.
Then $h_A^m e^{-\phi}$ is a singular Hermitian metric on $A^{\ot m}$
with semi-positive curvature.

We set 
$$
	L := \CO(ab+r)|_Z \ot \pi^*
		(A^{\ot b} \ot K_Y^{-1} \ot (\what \det F)^{-1})|_Z.
$$
We note $(\CO(ab) \ot \pi^*A^{\ot b})|_Z = K_Z \ot L$.
By Theorem \ref{MT}, $\CO(1)|_Z$ has a singular Hermitian metric 
$g_{\CO(1)}$ with semi-positive curvature.
Then the line bundle $L$ over $Z$ has a singular Hermitian metric
$$
  g_L := g_{\CO(1)}^{ab+r} \pi^*(h_A^{b-m-1} \cdot h_A^m e^{-\phi} 
			\cdot h_A h_{K_Y}^{-1} h_{\what \det F}^{-1})
$$
of semi-positive curvature.
Let $h_L$ be a smooth Hermitian metric on $L$.
Then $g_L$ can be written as $g_L = h_L e^{-\psi}$
for a function $\psi \in L^1_{loc}(Z, \BR)$, 
which is a sum of a smooth function and a plurisubharmonic
function around every point of $Z$.
Let $\ai\rd\rdb \psi = \ai(\rd\rdb \psi)_c + \ai(\rd\rdb \psi)_s$
be the Lebesgue decomposition into
the absolute continuous part $\ai(\rd\rdb \psi)_c$ and 
the singular part $\ai(\rd\rdb \psi)_s$.
We set $c(L, \psi) = \rdb\rd\log h_L + (\rd\rdb \psi)_c$.
Then $\ai c(L, \psi)$ is a semi-positive $(1,1)$-current, because 
it is the absolute continuous part of the curvature current of $g_L$.
We also note that $\ai c(L, \psi) \ge (b-m-1)\ai\,\rdb\rd \log (\pi^*h_A)$.

We take a section 
$\sg \in H^0(\BP(F_y), (\CO(ab) \ot \pi^*A^{\ot b})|_{\BP(F_y)})$,
and take a local extension 
$\sg' \in H^0(\BP(F|_{W_y}), \CO(ab) \ot \pi^*A^{\ot b})$.
We consider $u := \rdb ((\pi^*\rho) \sg') 
= \rdb (\pi^*\rho) \cdot \sg'$,
which can be seen as an $L$-valued $(p,1)$-form on $Z$,
where $p = \dim Z = m+r-1$.
At each point $z \in Z$, we set 
$|u|_{c(L, \psi)}^2(z) 
= \inf \{ \ga \in \BR_{\ge 0} \cup \{+\infty\} ;\ 
|(u, \gb)|^2 \le \ga^2 (\ai c(L, \psi) \Lambda \gb, \gb)$
for any $\gb \in \Omega_{Z, z}^{p,1} \ot L_z \}$ (see \cite[p.\ 468]{D82}).
Here $( \ , \ )$ is the Hermitian inner product of
$\Omega_Z^{p,1} \ot L$ with respect to $\w_Z$ and $h_L$, and
$\Lambda$ is the adjoint of the Lefschetz operator $\w_Z \wed \bullet$. 
Assume for the moment that
$\int_Z |u|_{c(L, \psi)}^2 e^{-\psi} dV < \infty$.
Then by \cite[5.1]{D82}, for $u$ with $\rdb u = 0$ and 
$\int_Z |u|_{c(L, \psi)}^2 e^{-\psi} dV < \infty$,
there exists $v \in L^2_{p,0}(Z, L, loc)$ (an $L$-valued $(p,0)$-form 
on $Z$ with locally square integrable coefficients)
such that $\rdb v = u$ and 
$\int_Z |v|^2 e^{-\psi} dV \le \int_Z |u|_{c(L, \psi)}^2 e^{-\psi} dV$.
Since $u \equiv 0$ on $\pi^{-1}(W_y')$,
$v$ is holomorphic on $\pi^{-1}(W_y')$.
The integrability $\int_Z |v|^2 e^{-\psi} dV < \infty$, 
in particular $\int_{\pi^{-1}(W_y)} |v|^2 e^{-\pi^* \phi} dV < \infty$
ensures $v|_{\BP(F_y)} \equiv 0$.
(In a modern terminology, the multiplier ideal sheaf 
$\CI(\pi^{-1}(W_y), e^{-\psi})$ is the defining ideal sheaf 
$\CI_{\BP(F_y)}$ of the fiber.) \
Then $\wtil \sg := (\pi^*\rho)\sg' - v \in H^0(Z, K_Z \ot L)$
and $\wtil \sg|_{\BP(F_y)} = \sg'|_{\BP(F_y)} = \sg$.

Let us see the integrability
$\int_Z |u|_{c(L, \psi)}^2 e^{-\psi} dV < \infty$.
Because $\Supp u \subset \pi^{-1}(W_y \sm W_y')$,
and $\psi$ is smooth on $\pi^{-1}(W_y \sm W_y')$,
it is enough to check that 
$|u|_{c(L, \psi)}^2 < \infty$ on $\pi^{-1}(W_y \sm W_y')$.
Let us take $z_0 \in Z$ such that $y_0 = \pi(z_0) \in W_y \sm W_y'$.
Let $(U, (z^1, \ldots, z^p))$ be a local coordinate centered at $z_0$
such that $dz^1, \ldots, dz^p$ form an orthonormal basis of $\Om^1_{Z}$
at $z_0$ so that 
$\w_Z = \frac{\ai}2 \sum_{i=1}^p dz^i \wed d\ol z^i$ at $z_0$.
Let $(y^1, \ldots, y^m)$ be a local coordinate centered at $y_0$.
We will use indexes $i, j$ (resp.\ $k, \ell$) for
$1, \ldots, p$ of $z^i$ (resp.\ $1, \ldots, m$ of $y^k$).
We have $\pi^*(dy^k) = \sum_{i=1}^p c^k_i dz^i$ at $z_0$, where
$c^k_i = \frac{\rd y^k}{\rd z^i}(z_0)$, and 
$\pi^*(\rdb \rho) 
= \sum_i (\sum_k \rho_{\ol k} \ol{c^k_i}) d\ol z^i$ at $z_0$,
where $\rho_{\ol k} = \frac{\rd \rho}{\rd \ol y^k}(y_0)$.
The canonical bundle $K_Z$ is trivialized by $dz = dz^1 \wed \ldots \wed dz^p$.
We take a nowhere vanishing section $e \in H^0(U, L)$ such that
$h_L(e,e)(z_0) = 1$.
Then we can write as
$u = \pi^*(\rdb \rho) \wed s dz \ot e$ with some $s \in H^0(U, \CO_Z)$.
We write the curvature form of $h_A$ as $\ai \Th_A 
= \frac{\ai}2 \sum_{k,\ell} a_{k\ol \ell} dy^k \wed d\ol y^\ell$
at $y_0$.
Then $\ai \pi^*\Th_A 
= \frac{\ai}2 \sum_{i,j} 
  (\sum_{k,\ell} a_{k\ol \ell} c^k_i \ol{c^\ell_j}) dz^i \wed d\ol z^j$
at $z_0$.
Let $\gb \in \Om^{p,1}_{Z,z_0} \ot L_{z_0}$, which is written as
$\gb = (\sum_i b_i dz \wed d\ol z^i) \ot e$.
We set $b^k = \sum_i c^k_i b_i$ for $1 \le k \le m$.
Since $\ai c(L, \psi) \ge \ai \pi^*\Th_A$, we have
$(\ai c(L, \psi) \Lambda \gb, \gb) 
\ge (\ai \pi^*\Th_A \Lambda \gb, \gb)
= 2^{p+1} \sum_{k, \ell} a_{k \ol \ell} b^k \ol{b^\ell}$.
Let $\lambda_1 > 0$ be the smallest eigenvalue of the positive matrix
$(a_{k \ol \ell})_{k, \ell}$.
Then $\sum_{k, \ell} a_{k \ol \ell} b^k \ol{b^\ell} 
\ge \lambda_1 \sum_{k} |b^k|^2$.
On the other hand 
$|(u, \gb)|^2 
= |(\sum_i (\sum_k \rho_{\ol k} \ol{c^k_i}) d\ol z^i \wed s dz \ot e, 
	(\sum_i b_i dz \wed d\ol z^i) \ot e)|^2
= (2^{p+1})^2 |s|^2 |\sum_k \rho_{\ol k} \ol{b^k}|^2$,
%
and we have
$|(u, \gb)|^2 \le 
(2^{p+1})^2 |s|^2 \sum_k |\rho_{\ol k}|^2 \sum_k |\ol{b^k}|^2$. 
Then
$|(u, \gb)|^2 \le
2^{p+1} \lambda_1^{-1} |s|^2 \sum_k |b^k|^2
(\ai c(L, \psi) \Lambda \gb, \gb)$.
We finally have
$$
|u|_{c(L, \psi)}^2(z_0) 
\le 2^{p+1} \lambda_1^{-1} |s|^2 \sum_k |b^k|^2 < \infty.
$$
Then the proof is complete.
\qed

\subsection{Variants} \label{Variants}

We shall give some variants of the results in the introduction.
In Theorem \ref{MT}\,(1), we need to restrict ourselves on 
a relatively compact subset $Y_0 \subset Y$
(see the proof of Lemma \ref{decomp} for the reason).
We remove it in some cases.

\begin{variant}
Let $f : X \lra Y$ be a proper surjective morphism with connected fibers
between smooth algebraic varieties, and let $(E, h)$ be 
a Nakano semi-positive holomorphic vector bundle on $X$.
Then the line bundle $\CO(1)$ 
for $\pi : \BP(\Rq|_{Y \sm S_q}) \lra Y \sm S_q$ 
has a singular Hermitian metric with semi-positive curvature, 
and which is smooth on $\pi^{-1}(Y \sm \Dl')$
for a closed algebraic subset $\Dl' \subsetneq Y$.
\end{variant}

\begin{proof}
By Chow lemma \cite[II.Ex.4.10]{Ha}, there exists a modification 
$\mu : X ' \lra X$ from a smooth algebraic variety $X'$ such that
$f' := f \circ \mu : X' \lra Y$ becomes projective.
Moreover by Hironaka, we may assume 
$\Supp {f'}^{-1}(\Dl')$ is simple normal crossing.
Here $\Dl' \subset Y$ is the discriminant locus of $f'$,
which $\Dl'$ may be larger than $\Dl$ for $f$.
Since a projective morphism is K\"ahler (\cite[6.2.i]{Tk}), 
we can take a relative K\"ahler form $\w_{f'}$ for $f'$.
We then have a Hodge metric on $\Rq|_{Y \sm \Dl'}$ with respect to
$\w_{f'}$ and $\mu^*h$.
The rest of the proof is the same as Theorem \ref{MT}, 
after Proposition \ref{MT'}.
\end{proof}

\begin{variant} \label{q=0}
Let $f : X \lra Y$ and $(E, h)$ be as in Set up \ref{basic},
and let $q = 0$.
Then, the line bundle $\CO(1)$ for 
$\pi : \BP(f_*(K_{X/Y} \ot E)|_{Y \sm S_0}) \lra Y \sm S_0$ 
has a singular Hermitian metric $g_{\CO(1)}$ 
with semi-positive curvature,
and whose restriction on $\pi^{-1}(Y \sm \Dl)$ is the quotient metric 
$\go$ of $\pi^*g$,
where $g$ is the Hodge metric with respect to $h$.
\end{variant}

\begin{proof}
In case $q = 0$, we have the Hodge metric $g$ on 
$f_*(K_{X/Y} \ot E)|_{Y \sm \Dl}$ 
with respect to $h$, which does not depend on a relative K\"ahler form.
This Hodge metric does not change, even if we take a modification 
$\mu : X' \lra X$ 
(more precisely, for any relatively compact open subset $Y_0 \subset Y$
and a modification $\mu : X_0' \lra X_0 = f^{-1}(Y_0)$)
which is biholomorphic over $X \sm f^{-1}(\Dl)$.
Once a global metric is obtained, the extension problem is 
a local issue.
Hence it is reduced to see that on every small coordinate neighbourhood
$Y_0 \subset Y$, $\go|_{\pi^{-1}(Y_0 \sm \Dl)}$ extends 
as a singular Hermitian metric 
on $\CO(1)|_{\pi^{-1}(Y_0)}$ with semi-positive curvature.
As we saw in the proof of Theorem \ref{MT}, this is reduced
to Theorem \ref{bdd} (or Theorem \ref{MT} itself).
\end{proof}

We have the following standard consequence of our theorems.
Corollary \ref{quotient} can be also formulated under other assumptions
as in two variants above. 
We left it for the readers.

\begin{cor} \label{quotient}
Let $f : X \lra Y$, $(E, h)$ and $0 \le q \le n$ be as 
in Set up \ref{basic}.
Let $L$ be a holomorphic line bundle on $Y$ with a surjection
$\Rq|_{Y \sm Z} \lra L|_{Y \sm Z}$ on the complement
of a closed analytic subset $Z \subset Y$ of $\codim_Y Z \ge 2$.

(1) Unpolarized case.
For every relatively compact open subset $Y_0 \subset Y$,
$L|_{Y_0}$ has a singular Hermitian metric with semi-positive curvature. 

(2) Polarized case.
Assume the simple normal crossing condition in Theorem \ref{MT}\,(2), 
and let $\w_f$ be a relative K\"ahler form  for $f$.
Then $L$ has a singular Hermitian metric with semi-positive curvature,
whose restriction on $Y \sm \Dl$ is the quotient metric of 
the Hodge metric $g$ on $\Rq|_{Y \sm \Dl}$ with respect to $\w_f$ and $h$.
\end{cor}

\begin{proof}
(1)
Denote by $F = \Rq$.
We put a Hermitian metric $g$ on $F|_{Y_0 \sm \Dl}$
as in Proposition \ref{MT'}.
Assume for the moment $S_q = Z = \emptyset$.
Then the line bundle $L$ corresponds to a section $s : Y \lra \BP(F)$ of 
$\pi : \BP(F) \lra Y$ such that $L \cong s^*\CO(1)$.
Moreover the Hodge metric $g$ on $F|_{Y_0 \sm \Dl}$ induces 
a quotient metric $g_L^\circ$ 
(resp.\ $\go$) of $L|_{Y_0 \sm \Dl}$ by quotient $F \lra L$ 
(resp.\ $\CO(1)|_{\pi^{-1}(Y \sm \Dl)}$ by $\pi^*F \lra \CO(1)$),
and $g_L^\circ = s^*\go$ over $Y_0 \sm \Dl$ by the definition.
Let $g_{\CO(1)}$ be the extension of $\go$ as a singular Hermitian 
metric on $\CO(1)|_{\pi^{-1}(Y_0)}$ with semi-positive curvature.
Then $g_L = s^* g_{\CO(1)}$ over $Y_0$
is a (unique) extension of $g_L^\circ$ with semi-positive curvature.

In case $S_q \cup Z$ may not be empty, by virtue of Hartogs type 
extension as in the proof of Theorem \ref{MT},
we can extend further the singular Hermitian metric $g_L$ 
on $L|_{Y_0 \sm (S_q \cup Z)}$ with semi-positive curvature 
as a singular Hermitian metric on 
$L|_{Y_0}$ with semi-positive curvature.
(2) is similar.
\end{proof}





\baselineskip=14pt

\vskip 10pt

Christophe Mourougane

\vskip 5pt

Institut de Recherche Math\'ematique de Rennes

Campus de Beaulieu

35042 Rennes cedex, France

e-mail: christophe.mourougane@univ-rennes1.fr

\vskip 10pt

Shigeharu Takayama

\vskip 5pt

Graduate School of Mathematical Sciences

University of Tokyo

3-8-1 Komaba, Tokyo

153-8914, Japan

e-mail: taka@ms.u-tokyo.ac.jp

\end{document}